\documentclass[11pt,thmsa]{article}

\usepackage{amssymb,latexsym}
\usepackage{pb-diagram}
\usepackage{amsmath,amscd}
\usepackage[all]{xy}

 \usepackage[latin1]{inputenc}
 \usepackage{amsmath}
 \usepackage{amsthm}
 \usepackage{amsfonts}
 \usepackage{amssymb}
 \usepackage[pdftex]{graphicx}
 \usepackage{wrapfig}
 \usepackage{layout}
 \usepackage{verbatim}
 \usepackage{alltt}
 \usepackage{xypic}
 \usepackage{bbm}
 \usepackage{tikz}
 \usetikzlibrary{matrix,arrows,decorations.pathmorphing}
 \input xy
 \xyoption{all}
 \usepackage{ upgreek }
 \usepackage{tikz-cd}
\usetikzlibrary{cd}
\usepackage{multicol}

 \DeclareMathAlphabet{\mathpzc}{OT1}{pzc}{m}{it}

 \newtheorem{theorem}{Theorem}[section]
 \newtheorem{lemma}[theorem]{Lemma}
 \newtheorem{proposition}[theorem]{Proposition}
 
 \newtheorem{corollary}[theorem]{Corollary}
 \newtheorem{definition}[theorem]{Definition}
 \newtheorem*{theoremnn}{Theorem}
 \newtheorem*{conjecturenn}{Conjecture}
  \theoremstyle{definition}
 \newtheorem{example}[theorem]{Example}
 \newtheorem{exercise}{Exercise}
 \newtheorem{remark}[theorem]{Remark}

\renewenvironment{proof}{\noindent{\it
Proof.}}{\bgroup\hspace{\stretch{1}}$\square$\egroup\medskip\par}

\newcommand{\id}{\mathrm{id}}

\newcommand{\R}{\mathbb{R}}

\newcommand{\MC}{\ensuremath{\mathbb C}}
\begin{document}

\def\tetra{\draw(A) -- (B) -- (C) -- (D) -- cycle; \draw[shade](A)--(B)--(D); \draw[shade](B)--(C)--(D);\draw(B)--(D); \draw[dotted](A)--(C);\draw[dotted](A)--(C); \draw(A)--(D);}
\def\tetraA{\draw(A) -- (B) -- (C) -- (D) -- cycle; \draw[shade=gray](B)--(C)--(D);\draw[thin](B)--(D); \draw[dotted](A)--(C);\draw[dotted](A)--(C); \draw[thick](A)--(D);}
\def\tetraB{\draw(A) -- (B) -- (C) -- (D) -- cycle; \shade(A)--(C)--(D);\draw(B)--(D); \draw(A)--(D); \draw[dotted](A)--(C);\draw(C)--(D);}
\def\tetraC{\draw(A) -- (B) -- (C) -- (D) -- cycle; \draw[shade=gray](A)--(B)--(D);\draw(B)--(D); \draw[dotted](A)--(C); \draw[dotted](A)--(C); \draw(A)--(D);\draw[thick](B)--(C);}
\def\tetraD{\draw(A) -- (B) -- (C) -- (D) -- cycle; \draw[shade=gray](A)--(B)--(C);\draw(B)--(D); \draw[dotted](A)--(C);\draw[dotted](A)--(C);\draw(A)--(D);  }

\def\triangle{\draw(A) -- (B) -- (C)--cycle; \draw[shade](A)--(B)--(C);\draw(A)--(C);}
\def\triangleA{\draw(A) -- (B) -- (C)  -- cycle;\draw[ultra thick](A)--(C);\draw[ultra thick](C)--(B);}
 \def\triangleB{\draw(A) -- (B) -- (C)  -- cycle;\draw[ultra thick](A)--(B);}


\def\line{\draw[thick](A) -- (B);\fill[ultra thick](A);}


\def\bleft{\draw(X) -- (Y) -- (Z)--(W);}
\def\bright{\draw(x) -- (y) -- (z)--(w);}

\vspace{15cm}
 \title{Lectures on the Euler characteristic affine manifolds }
\author{Camilo Arias Abad\\
Sebasti\'an V\'elez V\'asquez\footnote{Universidad Nacional de Colombia, Medellín.  camiloariasabad@gmail.com, svelezv@unal.edu.co.
} }

 \maketitle
 \begin{abstract}
These are lecture notes prepared for the summer school \emph{Geometric, algebraic and topological methods in quantum field theory}, held in Villa de Leyva in July 2017. Our goal is to provide an introduction to a conjecture of Chern that states that the Euler characteristic of a closed affine manifold vanishes. We present part of the history and motivation for the conjecture as well as some recent developments. All comments and corrections are most welcome!\footnote{Camilo would like to thank the organisers and participants of the school for a great time at Villa de Leyva.}\footnote{We thank Alexander Quintero V\'elez for his valuable comments and corrections on a preliminary version of the notes.}
 \end{abstract}

\tableofcontents

\section{Introduction}
An affine structure on a manifold is an atlas whose transition functions are affine transformations. The existence of such a structure is equivalent to the existence of a flat torsion free connection on the tangent bundle. Chern's conjecture states the following:

\begin{conjecturenn}[Chern $\sim$ 1955]
 The Euler characteristic of a closed affine manifold is zero.
 \end{conjecturenn}
 In case the connection $\nabla$ is the Levi-Civita connection of a riemannian metric, the Chern-Gauss-Bonnet formula:
\[ \chi(M)= \Big(\frac{1}{2\pi}\Big)^n \int_M \mathsf{Pf}(K)\]
implies that the Euler characteristic is zero. However, not all flat torsion free connections on $TM$ admit a compatible metric, and therefore, Chern-Weil theory cannot be used in general to write down the Euler class in terms of the curvature.\\

These notes contain an exposition of the following results concerning  the Euler characteristic of affine manifolds:
\begin{itemize}
\item In 1955, Benz\'ecri \cite{B} proved that a closed affine surface has zero Euler characteristic.

\item In 1958, Milnor \cite{Milnor} proved inequalities which completely characterise those oriented rank two bundles over a surface that admit a flat connection.

\item In 1975, Kostant and Sullivan \cite{KS} proved Chern's conjecture in the case where the manifold is complete.

\item In 1977, Smillie \cite{Smillie} proved that the condition that the connection is torsion free matters. For each even dimension greater than $2$, Smillie constructed closed manifolds with non-zero Euler characteristic that admit a flat connection on their tangent bundle.
\item In 2015, Klingler \cite{K} proved the conjecture for special affine manifolds. That is, affine manifolds that admit a parallel volume form.

\end{itemize}
As far as we can tell, the general case of Chern's conjecture remains open.\\ \\
We have added appendices with preliminary material.
Appendix A contains a review of the theory of connections, curvature and the Chern-Gauss-Bonnet theorem.
Appendix B contains an introduction to spectral sequences and Appendix C to sheaf theory.
\subsection{Affine manifolds}

\begin{definition}
A diffeomorphism $\varphi$ between open subsets of $\R^m$ is affine if it has the form
\[ \varphi(x)= Ax+b,\]
where $A \in {\mathrm{G}\mathrm{L}}(m,\R)$ and $b \in \R^m$.
\end{definition}

\begin{definition}
An affine structure on a manifold is an atlas such that all transition functions are affine and it is maximal with this property.
An affine manifold is a manifold together with an affine structure.
\end{definition}

It is possible to characterise affine structures in a more intrinsic manner that can be expressed without reference to an atlas, as the following lemma shows:

\begin{lemma}
Let $M$ be a manifold. There is a natural bijective correspondence between affine structures on $M$ and flat torsion free connections on $TM$.
\end{lemma}
\begin{proof}
Let $(U_\alpha,\varphi_\alpha)$ be an affine structure on $M$. There is a unique connection $\nabla$ on $TM$ whose Christoffel symbols vanish in affine coordinates. Conversely, given a flat torsion free connection $\nabla$ on $TM$, the set of coordinates for which the Christoffel symbols of $\nabla$ vanish gives an affine structure on $M$.
\end{proof}

\begin{example}
The torus $\mathbb{T}^m:= \R^m/\mathbb{Z}^m$ has a natural affine structure for which the projection map $\pi: \R^m \rightarrow \mathbb{T}^m$ is an affine local difeomorphism.
\end{example}

\begin{example}[Hopf manifolds]
Let us fix a real number $\lambda>1$ and consider the action of the group $\mathbb{Z}$ on $\R^m-\{0\}$ given by:
\[ n \star x:= \lambda^n x.\]
Since the action is free and proper, the quotient is a smooth manifold called the Hopf manifold $\mathsf{Hopf}^m_\lambda$. Since the group $\mathbb{Z}$ acts by affine transformations the quotient space is an affine manifold. Topologically, these manifolds are the union of two circles for $m=1$ and diffeomorphic to $S^{m-1}\times S^1$ for $m>1$.
\end{example}

\begin{definition}
An affine structure on a Lie group $G$ is called left invariant if for all $g \in G$:
\[ (L_g)^*(\nabla)=\nabla,\]
where $L_g$ denotes the diffeomorphism given by left multiplication by $g$.
\end{definition}

\begin{definition}
Let $V$ be a real vector space. A bilinear map:
\[\beta: V \otimes V \rightarrow V;\,\, v \otimes w \mapsto v\cdot w\]
is called left symmetric if:
\[ v\cdot (w \cdot z)-(v\cdot w) \cdot z= w\cdot (v \cdot z)-(w\cdot v)\cdot z.\]
\end{definition}

\begin{definition}
An affine structure on a finite dimensional real Lie algebra $\frak{g}$ is a left symmetric bilinear form on $\frak{g}$ such that for all $v,w \in \frak{g}$:
\[ [v,w]=v\cdot w- w \cdot v.\]
\end{definition}

\begin{lemma}
Let $G$ be a Lie group with Lie algebra $\frak{g}$. There is a natural bijective correspondence between
left invariant affine structures on $G$ and affine structures on $\frak{g}$.
\end{lemma}

\begin{proof}
For any $v \in \frak{g}=T_eG$ denote by $\widehat{v}$ the corresponding left invariant vector field. Given a left invariant affine structure on $G$ we define a bilinear form on $\frak{g}$ by:
\[ v\cdot w:= \nabla_{\widehat v}\widehat{w}(e).\]
Since $\nabla$ is torsion free we know that:
\[ v \cdot w - w \cdot v =  \nabla_{\widehat v}\widehat{w}(e)-  \nabla_{\widehat w}\widehat{v}(e) =[v,w].\]
Using the fact that $\nabla$ is flat and left invariant, we compute:
\begin{eqnarray*}
 v\cdot (w \cdot z)-(v\cdot w) \cdot z-w\cdot (v \cdot z)+(w\cdot v)\cdot z&=& \nabla_{\widehat{v}} (\nabla_{\widehat{w}}(\widehat{z}))-\nabla_{\nabla_{\widehat{v}}(\widehat{w})}(\widehat{z}) - \nabla_{\widehat{w}} (\nabla_{\widehat{v}}(\widehat{z}))+\nabla_{\nabla_{\widehat{w}}(\widehat{v})}(\widehat{z})\\
 &=&\nabla_{[\widehat{v},\widehat{w}]}(\widehat{z})-\nabla_{[\widehat{v},\widehat{w}]}(\widehat{z})\\
 &=&0.
\end{eqnarray*}
We conclude that the bilinear form is left symmetric. Conversely, given an affine structure on $\frak{g}$ there is a unique left invariant connection $\nabla$ such that:
\[ \nabla_{\widehat{v}}\widehat{w}(e)=v\cdot w.\]
The computations above show that this connection is flat and torsion free.
\end{proof}

\begin{example}
The Lie algebra $\frak{gl}(n,\R)$ admits a natural affine structure given by:
\[ A \cdot B:= AB.\]
We conclude that the Lie group $\mathrm{G}\mathrm{L}(n, \R)$ admits a left invariant affine structure.
\end{example}

The question of which Lie groups admit left invariant affine structures is an interesting problem. In \cite{Milnor2}, Milnor asked whether solvable Lie algebras admit affine structures. In \cite{Bu}, Burde showed that the answer to this question is negative.

\subsection{Complete manifolds and the developing map}

In this section we introduce the developing map of a simply connected affine manifold and use it to characterise complete affine manifolds.

\begin{theorem}
Let $M$ be an affine manfiold of dimension $m$ and $G$ be the group $\mathsf{Aff}(\mathbb{R}^m)= \mathrm{G}\mathrm{L}(n, \R)\ltimes \R^n$ seen as a discrete group. There is a natural principal $G$-bundle $\pi:\tau(M) \rightarrow M$ such that sections  of $\pi$ are in natural bijective correspondence with affine immersions from $M$ to $\R^m$. \end{theorem}
\begin{proof} For each $p\in M$ we define
	\[C_{p}:= \left\{  \varphi: U\to V\subseteq \R^m\mid  \varphi \text{ is an affine diffeomorphism and }p\in U \right\}\]
There is an equivalence relation $\sim$ on $C_p$ given by declaring that $\varphi \sim \varphi'$ if an only if there exists an open subset $W \subseteq U \cap U'$ such that $p \in W$ and $\varphi\vert_W=\varphi' \lvert_W$. Let us denote by $L_p$ the set of equivalence classes of elements in $C_p$ and set
	 \[\tau(M):=\coprod_{p} L_{p}.\]
	There is a natural map:
 $\pi:\tau(M) \rightarrow M$
 that sends $L_p$ to $p$.
 The Lie group $G$ acts on each set $L_p$ by composition:
 \[g \cdot \varphi := g \circ \varphi,\]
 and this action is free and transitive. Moreover, an affine chart $\varphi: U \rightarrow V\subseteq \R^m$ induces a natural identification:
\[ \widehat{\varphi}: \mathsf{Aff}(\R^m)\times U \rightarrow \pi^{-1}(U); \,\, (g,p)\mapsto [ g \circ \varphi]_p, \]
where $[g \circ \varphi]_p$ denotes the class of $g \circ \varphi$ in $L_p$. There is a unique topology on $\tau(M)$ such that $\widehat{\varphi}$ is a homeomorphism for all affine diffeomorphisms $\varphi$. The map $\pi: \tau(M) \rightarrow M$ is a principal $G$ bundle with respect to this topology. Let $\sigma$ be a section of $\pi$. Then we define $\widetilde{\sigma}:M \rightarrow \R^m$ by:
\[ \widetilde{\sigma}(p):= \sigma(p) (p).\]
By construction, the map $\widetilde{\sigma}$ is an affine immersion. Conversely, an affine immersion $f: M \rightarrow \R^m$ defines a section $\sigma_f$ given by:
\[ \sigma_f(p):= [f]_p.\]
\end{proof}

\begin{corollary}\label{extend}
Let $M$ be a simply connected affine manifold. Any affine chart $\varphi: U \rightarrow V \subseteq \R^m$ extends uniquely to an affine immersion from $M$ to $\R^m$. \end{corollary}

\begin{proof}
Since $M$ is simply connected, the covering space $\pi: \tau(M) \rightarrow M$ is trivial and therefore any local section extends uniquely to a global one.
\end{proof}

\begin{definition}
A developing map for an affine simply connected manifold is an affine immersion into $\R^m$.
 \end{definition}
	
\begin{corollary}
	If $M$ is an affine connected manifold with finite fundamental group then $M$ is not compact.
\end{corollary}

\begin{proof}
If $M$ is compact then so is its universal cover $\widetilde{M}$ which is simply connected and therefore admits an immersion to $\R^m$. This is impossible.
\end{proof}

\begin{definition}
An affine manifold $M$ with affine connection $\nabla$ is called complete if all geodesics can be extended to arbitrary time.
\end{definition}

\begin{lemma}\label{unique}
Let $M$ and $N$ be connected affine manifolds of the same dimension and $f,g:M \rightarrow N$ affine immersions. If $f$ and $g$ coincide on a nonempty open subset of $M$ then they are equal.
\end{lemma}
\begin{proof}
First we observe that the lemma is true if $M$ and $N$ are connected open subsets of $\R^m$.
For the general case, let $X\subset M$ be the subset of $M$ that consists of points $p\in M$ such that $f$ and $g$ coincide in an open neighborhood of $p$. Clearly, $X$ is open. It suffices to show that $X$ is closed.
Fix $p \in X$ and $q \in M$ arbitrary. Since $M$ is connected, there is a path $\gamma: I \rightarrow M$ such that
$\gamma(0)=p$ and $\gamma(1)=q$. Set \[Y:= \gamma^{-1}(M-X).\] We would like to show that $Y=\emptyset$. Suppose that $Y$ is nonempty and set $y:=\inf(Y)$ and $z=\gamma(y)$. Choose affine coordinates with connected domains $\varphi: U \rightarrow V$ around $z$ and $ \psi: W \rightarrow Z $ around $f(z)$ such that $f(U) \subseteq W$ and $g(U) \subseteq W$. Since $p$ is in $X$ there exists $l<y$ such that $\gamma(l) \in U\cap X$. Then $f\vert_U$ and $g\vert_U$ coincide in an open neighborhood. Since both $U$ and $ W$ are isomorphic to open subsets of $\R^m$ we conclude that $f\vert_U=g\vert_U$. This is a contradiction.
\end{proof}

\begin{lemma}\label{local}
Let $M$ be a simply connected complete affine manifold of dimension $m$. Then any affine coordinate $\varphi: U \rightarrow V\subseteq \R^m$ can be extended uniquely to a diffeomorphism from $M$ to $\R^m$
\end{lemma}
\begin{proof}
Fix a point $p \in M$ and affine coordinates $\varphi: U \rightarrow V$. By Corollary \ref{extend},
$\varphi$ can be extended uniquely to an affine immersion $\widetilde{\varphi}: M \rightarrow \R^m$. We will prove that $\widetilde{\varphi}$ is a diffeomorphism. Since $M$ is complete, the exponential map is defined on the whole tangent space $T_pM$:
\[ \mathsf{Exp}: T_pM \rightarrow M.\]
On the other hand, the derivative of $\varphi^{-1}$ at $x=\varphi(p)$ is an isomorphism from $\R^m$ to $T_pM$. We claim that
the map: \[\psi(y):=\mathsf{Exp} \circ D_x(\varphi^{-1})(y-x)\]
is the inverse of $\widetilde{\varphi}$. By Lemma \ref{unique}, it suffices to show that the functions are inverses to each other in a small open neighborhood. For $q \in U$ we compute:
\[ \psi (\widetilde{\varphi}(q))= \mathsf{Exp} \circ D_x(\varphi^{-1})(\varphi(q)-x)= \varphi^{-1}\circ \mathsf{Exp}(\varphi(q)-x))=q.\]
Conversely, for $y \in V$ we have:
\[ \widetilde{\varphi}\circ \psi(y)= \varphi \circ \mathsf{Exp} \circ  D_x(\varphi^{-1})(y-x)=\mathsf{Exp}\circ D \varphi \circ D_x(\varphi^{-1})(y-x) =y. \]
\end{proof}

\begin{proposition}\label{develop}
	Let $M$ be an affine manifold. The following statements are equivalent:
	\begin{enumerate}
		\item  $M$ is complete.
		\item There is an affine diffeomorphism $M \cong \R^m /\Gamma$, where $\Gamma$ is a discrete subgroup of $\mathsf{Aff}(\R^m)$ and the projection $\pi: \R^m \rightarrow M$ is the universal cover of $M$.
	\end{enumerate}
\end{proposition}

  \begin{proof}
  	Let us assume that $M$ is complete. Then $\widetilde{M}$ is also complete and simply connected. By Proposition \ref{develop} there is an affine difeomorphism $\varphi: \widetilde{M} \rightarrow \R^m$. Let us fix a point $p \in M$. Then there is an action of $\pi_1(M,p)$ on $\widetilde{M}$ and therefeore on $\R^m$. Since this action is by affine transformations, it defines a homomorphism $\rho: \pi_1(M,p) \rightarrow \mathsf{Aff}(\R^m)$. If we set $\Gamma:= \mathsf{Im}(\rho)$ we get by construction that $ M \cong \R^m/\Gamma$. Conversely, any manifold of the form $\R^m/\Gamma$ is complete.

  \end{proof}

\section{Some classical results}

\subsection{Milnor's inequalities and the conjecture in dimension $d=2$}
Throughout this section we will set $G=\mathrm{G}\mathrm{L}^{+}(2,\R)$, $S=\mathrm{S}\mathrm{O}(2)=\mathrm{U}(1)$ and $\Sigma$ will be a closed oriented surface of positive genus $g$.
Oriented rank two vector bundles over $\Sigma$ are classified by their Euler class. Milnor considered the question
of which of those bundles admit a flat connection. The answer is given by his famous inequality.

\begin{definition}
Let $\pi: E \rightarrow \Sigma$ be a rank two oriented vector bundle. The degree of $E$ is the number
\[ D(E)= \int_\Sigma e(E).\]
Here $e(E)$ denotes the Euler class of $E$.
\end{definition}
The Theorem we aim to prove is the following
\begin{theoremnn}
A rank two oriented vector bundle $\pi:E \rightarrow \Sigma$ admits a flat connection if and only if
\[ |D(E)|<g.\]
\end{theoremnn}
Milnor's inequality gives, in particular, a positive answer to Chern's conjecture in dimension $d=2$, a result that had been previously obtained by Benzecri \cite{B}.\\

The classification of rank two oriented bundles is of course the same as the classification of $\mathrm{GL}^{+}(2,\R)$ principal bundles.

\begin{lemma}
The inclusion $\iota: S \hookrightarrow G$ is a homotopy equivalence and therefore the classifying spaces $BS$ and $BG$ are homotopy equivalent.
\end{lemma}
\begin{proof}
	By the spectral theorem every nonsingular matrix $\gamma$ can be written uniquely in the form $\gamma=OR$, where $O$ is orthogonal matrix and $R$ is symmetric positive definite matrix. Then the following function defines a retraction:
	$$
	\begin{array}{rcc}
	\theta :G & \longrightarrow & S\\
	\gamma & \longmapsto & O
	\end{array}
	$$
\end{proof}

\begin{exercise}\label{trace}
Show that if
\[A= \left( {\begin{array}{cc}
			a & b\\
			c & d \\
			\end{array} } \right) \in G,\] then
		\[\theta(A)=\frac{1}{\sqrt{x^{2}+y^{2}} } \left( {\begin{array}{cc}
				x & y\\
				-y &  x \\
				\end{array} } \right). \]
			Here $x=a+d$ and $y=b-c$.
Moreover, show that:			
\begin{enumerate}
\item If $A \in G$ has positive trace, then $\theta(A)$ has positive trace.
\item If $A$ and $B$ are symmetric and positive definite matrices, then $AB$ has positive trace.
\item If $A,B \in G$ are symmetric and positive definite matrices, then $\theta(A B)$ has positive trace.
\item $\theta(A^{-1})=\theta(A)^{-1}$.
\end{enumerate}
\end{exercise}

\begin{lemma}
Let $\mathsf{Bun}(\Sigma)$ be the set of isomorphism classes of rank two oriented vector bundles over $\Sigma$. The degree map
$$
	\begin{array}{rccc}
	D:& \mathsf{Bun}(\Sigma) &\longrightarrow &\mathbb{Z}\\
	&E &\longmapsto &D(E)
	\end{array}
	$$
is a bijection.
\end{lemma}
\begin{proof}
The set $\mathsf{Bun}(\Sigma)$ is in bijective correspondence with the set of isomorphism classes of principal $G$ bundles over $\Sigma$. These are
classified by the set $[\Sigma, BG]$ of homotopy classes of maps to $BG$. Moreover:
\[[\Sigma,BG]\cong[\Sigma,BS]\cong[\Sigma,\mathbb{CP}^{\infty}]=[\Sigma,K(\mathbb{Z},2)]\cong H^2(\Sigma,\mathbb{Z}) \cong \mathbb{Z}.\]
By definition of the Euler class, this bijection is given by the degree.
\end{proof}

\begin{exercise}
Fix a point $p \in \Sigma$ and a homomorphism $\rho: \pi_1(\Sigma, p) \rightarrow G$. The group $\pi=\pi_1(\Sigma,p)$ acts on the right on the universal cover $\widetilde{\Sigma}$ and on the left on $\R^2$ via the representation $\rho$. Let $\widetilde{\Sigma}\times_\rho \R^2$ be the quotient of $\widetilde{\Sigma} \times \R^2$ via the equivalence relation
\[ (m g,v)\sim (m,\rho(g)v).\]
Show that $\widetilde{\Sigma}\times_\rho \R^2$ has a natural structure of a flat oriented vector bundle over $\Sigma$. We will denote this vector bundle $E^\rho$ and call it the bundle associated to $\rho$.
\end{exercise}

\begin{lemma}
A vector bundle over $\Sigma$ admits a flat connection if an only if it is isomorphic to one of the form $E^\rho$.
\end{lemma}

\begin{proof}
By the previous exercise, $E^\rho$ admits a flat connection. Assume that $E$ admits a flat connection $\nabla$ and fix a point $p \in \Sigma$ and an isomorphism $E_p\cong \R^2$. The holonomy of $\nabla$ gives a representation $\rho: \pi_1(\Sigma, p) \rightarrow G$ and an isomorphism $E \cong E^\rho$.
\end{proof}

An oriented closed surface $\Sigma$ of genus $g$ admits a CW structure with one zero-dimensional cell, $2g$  one-dimensional cells and one two-dimensional cell. For example, the surface of genus $2$ can be obtained by the following identification:
\begin{center}
\includegraphics[scale=0.5]{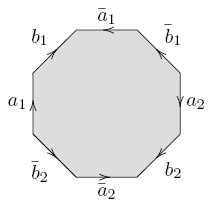}
\end{center}
In general, this decomposition gives the following presentation of the fundamental group of $\Sigma$:
\[\pi_1(\Sigma)=\bigg{\langle} a_1,\dots, a_g, b_1, \dots, b_g  \,\big\vert\,  \prod_{i=1}^g [a_i,b_i]=1 \bigg{\rangle} \]
Here $[a,b]$ denotes the commutator $aba^{-1}b^{-1}$. Thus, a representation $\rho$ of $\pi_1(\Sigma)$ on $G$ is the same as a choice of matrices
$A_1,\dots,A_g,B_1, \dots B_g \in G$ such that \[\prod_{i=1}^g [A_i,B_i]=\mathsf{id}.\]
We have seen that such a representation gives rise to a vector bundle $E^\rho$ and the natural question arises of computing the degree of $E^\rho$ in terms of
the matrices $A_i, B_i$.\\
Let $\widetilde{G}$ be the universal cover of $G$ and consider the short exact sequence
\[0 \rightarrow \pi_1(G) \xrightarrow{\iota} \widetilde{G} \xrightarrow{p} G \rightarrow 0.\]
Consider also the map $\phi: \R \rightarrow S$ given by:
$$\begin{array}{rcc}
 \alpha & \longmapsto & \phi(\alpha)=\left[ {\begin{array}{cc}
                              	\cos \alpha & \sin \alpha\\
                             	-\sin \alpha & \cos \alpha \\
                           	\end{array} } \right]
\end{array}$$
The map $\phi$ is the universal covering map of $S$ and therefore the retraction map $\theta: G \rightarrow S$ lifts naturally to a map:
\[ \widetilde{\theta}: \widetilde{G} \rightarrow \R.\]
Thus we obtain the following commutative diagram:
\begin{center}
\begin{tikzcd}
	& 0 \arrow[d] & 0 \arrow[d] &  \\
	& \pi_1(G) \arrow[r] \arrow[d,"\iota"] & \mathbb{Z} \arrow[d, "2 \pi"] & \\
	& \widetilde{G} \arrow[r, "\widetilde{\theta}"] \arrow[d,"p"] & \mathbb{R}  \arrow[d,"\phi"] &\\
	\pi_1(\Sigma) \arrow [r, "\rho"]	&G \arrow[r,"\theta"]  & S &
\end{tikzcd}	
\end{center}
\begin{exercise}\label{center}
	Let $G$ be a connected Lie group with universal cover $p:\widetilde{G}\to G$. Show that if $g \in \widetilde{G}$ is such that $p(g) \in Z(G)$ then $g \in Z(\widetilde{G})$.
\end{exercise}
Given a homomorphism $\rho: \pi_1(\Sigma) \rightarrow G$ determined by matrices $A_1, \dots, A_g, B_1,\dots, B_g$, we define the number $\delta(\rho)$ by
\[ \delta(\rho):= \frac{1}{2\pi}\widetilde{\theta}\left(\prod_{i=1}^g [\alpha_i,\beta_i]\right).\]
Here $\alpha_i, \beta_i$ are elements in $\widetilde{G}$ such that $p(\alpha_i)=A_i$ and $p(\beta_i)=B_i$.
\begin{lemma}
The number $\delta(\rho)$ is an integer and depends only on the representation $\rho$.
\end{lemma}
\begin{proof}
Let us first prove that for any choice of the $\alpha_i$ and $\beta_i$ the number $\delta(\rho)$ is an integer. By construction, $\prod_{i=1}^g [\alpha_i,\beta_i] \in \mathsf{ker}(p)$. This implies that $\widetilde{\theta}(\prod_{i=1}^g [\alpha_i,\beta_i])$ is a multiple of $2\pi$ and the result follows. Let us now show that the number is well defined. Suppose that $\alpha_i'$ and $\beta_i'$ is a different choice. Then, for each $i$ we know that:
\[ \alpha_i'=\alpha_i x_i\;\, \beta_i'= \beta_i y_i ,\]
where the $x_i$ and the $y_i$ are in the kernel of $p$ which, by the previous exercise, is contained in the center of $\widetilde{G}$. This implies that:
\[ [\alpha_i, \beta_i]=[\alpha_i',\beta_i'],\]
and the result follows.
\end{proof}
\begin{lemma}\label{milnornumber}
Given a representation $\rho: \pi_1(\Sigma) \rightarrow G$ we have:
\[ \delta(\rho)=D(E^\rho).\]
\end{lemma}
The idea of the proof is to construct a section of $E^\rho$ with a single zero whose index is precisely $\delta(\rho)$. Hopf's Index Theorem then guarantees that $\delta(\rho)=D(E^\rho)$. The first step towards constructing the section will be to define a path homotopic to the one used in the definition of $\delta(\rho)$. The following two lemmas will serve to prove said homotopy.
\begin{lemma}\label{homotopy1}
Let $A,B\in G$ and $\alpha,\beta:[0,1]\to G$ be paths connecting the identity to $A$ and $B$ respectively. The paths $\alpha\beta$ and $A\beta\ast\alpha$ connecting the identity to $AB$ are homotopic.
\end{lemma}
\begin{proof}
  For every $s\in[0,1]$ define paths $\kappa_s,\sigma_s$ as follows:
\[
\kappa_s(t)=\alpha(s(t-1)+1),\qquad \sigma_s(t)=\alpha(t-st).
\]
Define the homotopy $H(s,t)=(\kappa_s\beta\ast\sigma_s)(t)$. Note that $\sigma_s(1)=\alpha(1-s)=\kappa_s(0)\beta(0)$, therefore the concatenation is well defined. One readily verifies that
\[
H(0,t)=(A\beta\ast\alpha)(t),\quad H(1,t)=\alpha(t)\beta(t),\quad H(s,0)=\mathsf{id},\quad H(s,1)=AB.
\]
\end{proof}
\begin{lemma}\label{homotopy2}
  Fix $v\in\mathbb{R}^2-\{0\}$. Consider the map $\sigma_v:G\to \mathbb{S}^1$ defined by \[\sigma_v(A)=\frac{Av}{||Av||}\]
Then all maps $\sigma_v$ are homotopic; furthermore, they are all homotopic to the map $\theta:G\to S\cong\mathbb{S}^1$.
\end{lemma}
\begin{proof}
  For the first part of the lemma take vectors $v,v'\in\mathbb{R}^2-\{0\}$ and a path $\gamma:[0,1]$ from $v$ to $v'$ such that $\gamma(t)\neq0$ for every $t$. Then $H(A,t)=A\gamma(t)/||A\gamma(t)||$ is a homotopy between $\sigma_v$ and $\sigma_{v'}$.\\

  For the second part we take the first element of the canonical basis $e=(1,0)$ and prove that $\sigma_{e}$ is homotopic to $\theta$. For this purpose we recall the explicit formula for $\theta$ proven in Exercise \ref{trace}:
\[
\theta(A)=\theta\left[\begin{array}{cc}a&b\\c&d\end{array}\right]=\frac{1}{\sqrt{(a+d)^2+(b-c)^2}}\left[\begin{array}{cc}a+d&b-c\\c-b&a+d\end{array}\right]
\]
The isomorphism $S\cong\mathbb{S}^1$ sends a matrix to its first column, therefore $\theta$ seen as a map with values in $\mathbb{S}^1$ is given by
\[
\theta(A)=\theta\left[\begin{array}{cc}a&b\\c&d\end{array}\right]=\frac{1}{\sqrt{(a+d)^2+(b-c)^2}}\left[\begin{array}{c}a+d\\c-b\end{array}\right]
\]
On the other hand, the formula for $\sigma_e$ is
\[
\sigma_e(A)=\sigma_e\left[\begin{array}{cc}a&b\\c&d\end{array}\right]=\frac{1}{\sqrt{a^2+c^2}}\left[\begin{array}{c}a\\c\end{array}\right]
\]
A homotopy between $\theta$ and $\sigma_e$ is then given by
\[
H_t(A)=H_t\left[\begin{array}{cc}a&b\\c&d\end{array}\right]=\frac{1}{\sqrt{(a+td)^2+(tb-c)^2}}\left[\begin{array}{c}a+td\\c-tb\end{array}\right]
\]
We check that $H_t$ is well defined. Suppose that $(a+td)^2+(tb-c)^2=0$, then we have $a=-td$ and $c=tb$ for some $t\in[0,1]$, whence $\det(A)=-t(b^2+d^2)\leq0$.
\end{proof}
Now we are ready to prove Lemma \ref{milnornumber}:\\

\begin{proof}
Recall that we wish to construct a section of $E^\rho$ with a single zero with index $\delta(\rho)$.\\
Consider the path $\gamma:[0,1]\to G$ given by the formula
\[
t\mapsto \gamma(t)=\prod_{i=1}^g[\alpha_i(t),\beta_i(t)].
\]
Clearly we have $\gamma(0)=\gamma(1)=\mathsf{id}$ and therefore the composition $\theta\circ \gamma$ induces a map $\widehat\gamma:\mathbb{S}^1\to\mathbb{S}^1$. Let $\widetilde \gamma:[0,1]\to\mathbb{R}$ be the unique lifting of $\theta\circ\gamma$ that begins at $0$. Considering the paths $\alpha_i,\beta_i,\gamma$ as elements of $\widetilde G$, we see that $\widetilde \gamma$ is actually given by the formula
\[
\widetilde \gamma=\widetilde\theta(\gamma)=\widetilde\theta\left(\prod_{i=1}^g[\alpha_i,\beta_i]\right)
\]
Since the degree of $\widehat\gamma$ is just its homotopy class, it can be calculated via the lifting $\widetilde \gamma$ as follows
\[
\text{deg}(\widehat\gamma)=\frac{1}{2\pi}\widetilde\theta\left(\prod_{i=1}^g[\alpha_i,\beta_i]\right)=\delta(\rho).
\]
Now let us define the section.
\begin{figure}[ht!]
\includegraphics[trim= 0mm 10mm 0mm 10mm, scale=1]{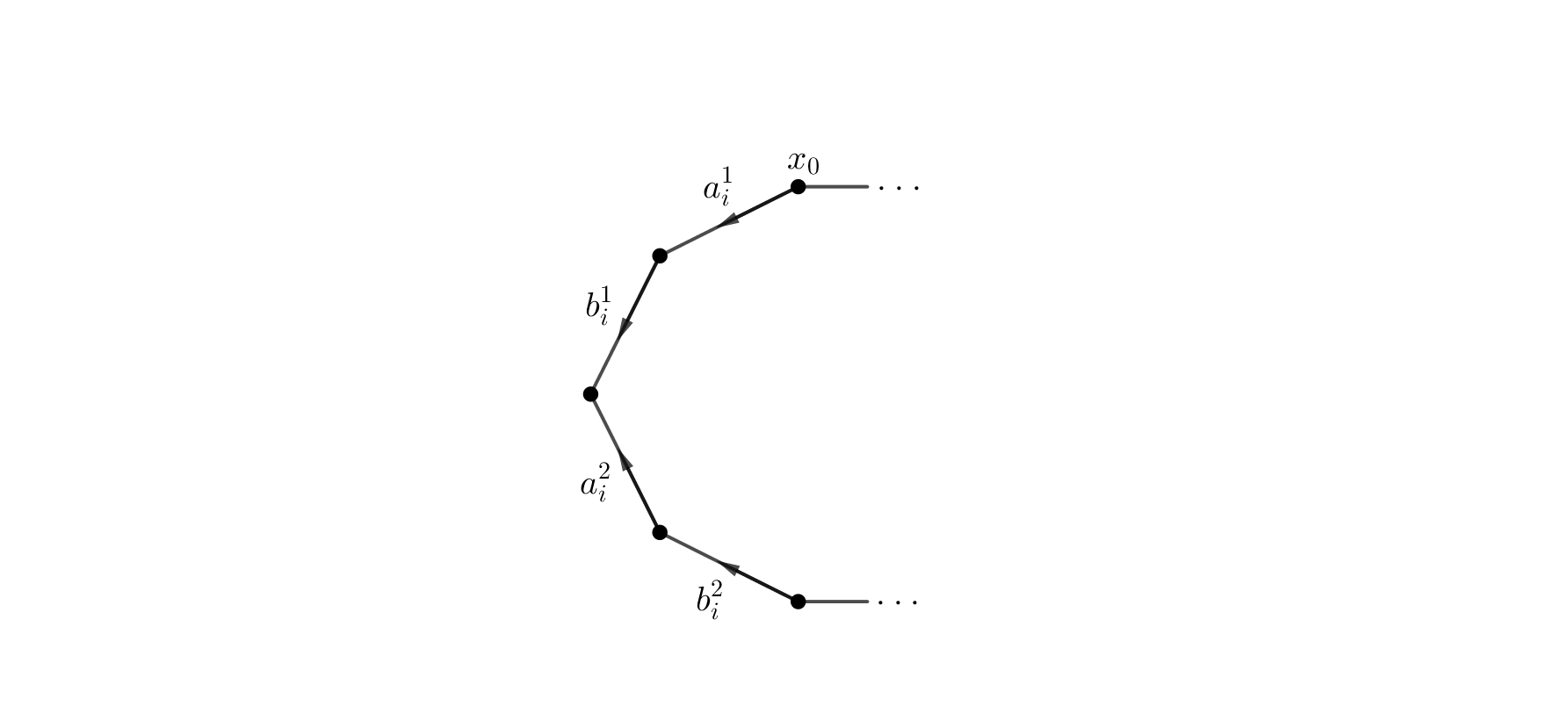}
\caption{Polygon}
\label{polygon}
\end{figure}
If $\Sigma$ has genus $g$, then it may be obtained from a polygon with $4g$ sides glued together as follows: suppose that we have a portion of the polygon as shown in Figure \ref{polygon}, then the sides are glued following the rule
\[
a_i^1(t)\sim a_i^2(t),\qquad b_i^1(t)\sim b_i^2(t),
\]
for $i=1,\cdots,g$. Notice that sides $a_i^1,a_i^2$ map to the generator $[a_i]$ of $\pi_1(\Sigma)$. Similarly, paths $b_i^1,b_i^2$ map to generator $[b_i]$.
\medskip
We first define the section $s:\Sigma\to \widetilde\Sigma\times\mathbb{R}^2$ in the boundary of the polygon using the liftings $\alpha_i,\beta_i$. Suppose that we have $s(x_0)=([x_0], v)$, where $[x_0]\in\widetilde\Sigma$ denotes the class of the constant path. Then we define $s$ by
\begin{align*}
  &s(a_i^1(t))=([a_{i,t}],\alpha_i^{-1}(t)v)\\
  &s(b_i^1(t))=([b_{i,t}\ast a_i], A_i^{-1}\beta_i^{-1}(t)v)\\
  &s(a_i^2(t))=([a_{i,t}\ast\bar a_i\ast b_i\ast a_i], A_i^{-1}B_i^{-1}A_i\alpha_i^{-1}(t)v)\\
  &s(b_i^2(t))=([b_{i,t}\ast\bar b_i\ast\bar a_i\ast b_i\ast a_i], A_i^{-1}B_i^{-1}A_iB_i\beta_i^{-1}(t)v).
\end{align*}
In the previous equations, $[a_{i,t}]$ and $[b_{i,t}]$ denote the classes of the paths running from $a_i(0)$ to $a_i(t)$ and from $b_i(0)$ to $b_i(t)$ respectively. Notice that
\[
s(a_i^1(t))\sim s(a_i^2(t)),\qquad s(b_i^1(t))\sim s(b_i^2(t)),
\]
which means that the section is well defined after gluing the sides of the polygon. Furthermore, if $v\neq 0$ then $s$ is non-zero. Next we extend the section to the whole polygon by identifying it with the unitary disk $\mathbb{D}$. Then $s$ may be extended to the interior of $\mathbb{D}$ by
\[
s(x)=\left\{\begin{array}{ccc}||x||s\left(\frac{x}{||x||}\right)&\text{if}& x\neq0\\0&\text{if}&x=0\end{array}\right.
\]
The index of the single zero of this section is the degree of the map obtained by multiplying a closed path $\lambda:[0,1]\to G$ with a nonzero vector $v$ and then normalizing. From the formula for the section $s$ we get the $i$-th portion of the path, which we denote $\lambda_i$, as a concatenation of paths:
\[
\lambda_i=A_i^{-1}B_i^{-1}A_i\beta\ast A_i^{-1}B_i^{-1}\alpha\ast A_i^{-1}\beta_i^{-1}\ast \alpha_i^{-1}
\]
By Lemma \ref{homotopy1}, $\lambda_i$ is homotopic to path $\alpha_i^{-1}\beta_i^{-1}\alpha_i\beta_i=[\alpha_i^{-1},\beta_i^{-1}]$. Gluing together the $\lambda_i$, we get that $\lambda$ is homotopic to the path $\gamma$. Finally, Lemma \ref{homotopy2} guarantees that the index of $s$ is $\text{deg}(\widehat\gamma)=\delta(\rho)$.
\end{proof}

\begin{lemma}\label{angle}
	The function $\widetilde{\theta}: \widetilde{G} \to \mathbb{R} $ satisfies the property:
	\[|\widetilde{\theta}(\alpha \beta)-\widetilde{\theta}(\alpha)-\widetilde{\theta}(\beta)|<\pi/2.\]
\end{lemma}
\begin{proof}
  	Given $A,B\in G$ there are symmetric positive definite matrices $R,T$ such that:  	
	\[AB=\theta(A)R \theta(B)T.\]
  	If we set $X=\theta(B)^{-1}R\theta(B)$ then:
  	\[AB=\theta(A)\theta(B)XT.\]
  	This implies that: \[\theta(AB)=\theta(A)\theta(B)\theta(XT).\]
	Therefore:
	\[\theta(B)^{-1}\theta(A)^{-1}\theta(AB)=\theta(XT)\]
	Since $X,T$ are positive definite, Exercise \ref{trace} implies that $\theta(B)^{-1}\theta(A)^{-1}\theta(AB)$ has positive trace.
	Fix $\alpha, \beta \in \widetilde{G}$ such that $p(\alpha)=A$ and $p(\beta)=B$ and
	set $\Delta=\widetilde{\theta}(\alpha \beta)-\widetilde{\theta}(\alpha)-\widetilde{\theta}(\beta)$. Then:
	\[ \phi(\Delta)=\left[ {\begin{array}{cc}
                              	\cos \Delta & \sin \Delta \\
                             	-\sin \Delta & \cos \Delta \\
                           	\end{array} } \right]=\theta(B)^{-1}\theta(A)^{-1}\theta(AB)\]
	has positive trace. So $\cos \Delta >0$. Since $\Delta$ is a continuous function of $\alpha$ and $\beta$ which vanishes when $\alpha=\beta=1$ we conclude that
	$|\Delta|<\pi/2$.

	  \end{proof}

\begin{lemma}\label{bound}
Given a representation $\rho: \pi_1(\Sigma) \rightarrow G$ we have:
\[ |\delta(\rho)|<g.\]
\end{lemma}

\begin{proof}
Applying Lemma \ref{angle} $4g-1$ times one obtains:
\[|\widetilde{\theta}([\alpha_1,\beta_1]\dots [\alpha_g,\beta_g] )-\widetilde{\theta}(\alpha_{1})-\widetilde{\theta}(\beta_{1})-\widetilde{\theta}(\alpha_{1}^{-1})-\widetilde{\theta}(\beta^{-1}_1) \dots-\widetilde{\theta}(\alpha_{g})-\widetilde{\theta}(\beta_{g})-\widetilde{\theta}(\alpha_{g}^{-1})-\widetilde{\theta}(\beta^{-1}_g)  |<(4g-1)\frac{\pi}{2}\]
By Exercise \ref{trace} we know that
\[\widetilde{\theta}(\alpha)+\widetilde{\theta}(\alpha^{-1})=0,\]
so that the inequality above becomes:
\[|\widetilde{\theta}([\alpha_1,\beta_1]\dots [\alpha_g,\beta_g])  |<(4g-1)\frac{\pi}{2}.\]

Dividing by $2\pi$ one obtains:
\[ |\delta(\rho)|<g-\frac{1}{4}<g.\]

\end{proof}

The lemma gives one direction of Milnor's inequalities. We will now show that the converse is also true.
Let us set $A_{0}:= \left( \begin{array}{cl}
	                               2 & 0 \\
	                              0 & 1/2
	                     \end{array} \right)  $                    and   $\alpha_{0} \in p^{-1}(A_{0})$ so that $\widetilde{\theta}(\alpha_{0})=0$.
         We will denote by $K$ and $\widetilde{K}$ the conjugacy classes of $A_{0}$ and $\alpha_{0}$ respectively. \\

         Since $\R$ is simply connected, the group homomorphism $\phi: \R \rightarrow S\subset G$ lifts to a homomorphism $\R \rightarrow \widetilde{G}$ and therefore
         $\R$ acts on $\widetilde{G}$ and also on conjugacy classes of $\widetilde{G}$.

\begin{lemma}\label{productmil}
	Any element of $\pi \widetilde{K}$ can be written as a product of two elements in $\widetilde{K}$.
\end{lemma}
\begin{proof}
	Consider the matrices:

	              \[A_{1}=\left(  \begin{array}{cl}
	                         -5/2 & 9/2\\
	                         -3 & 5
	              \end{array}\right)\] and
	          \[A_{2}=A_{0} A_{1} =
	              \left( \begin{array}{cl}
	                          -5 & 9\\
	                          -3/2 & 5/2
	
	              \end{array}\right).\]
Since the trace and determinant characterise the sets ${K}$ and $ \pi K$,
                                 then $A_{0}, A_{1} \in K$ and  $A_{2} \in \pi K$.
	            Let  $\alpha_{0},\alpha_{1} \in \widetilde{K} $ be such that $p(\alpha_i)=A_i$ and $\widetilde{\theta}(\alpha_{i})=0$. Then $\alpha_{0} \alpha_{1} \in  n\pi \widetilde{K}$ for some odd integer $n$. Let us show that $n=\pm1$.
	            Given  $\alpha \in \widetilde{K}$ we know that $p(\alpha)$ has positive trace and therefore $\cos(\widetilde{\theta}(\alpha))>0$.
Therefore, if $\beta \in n\pi \widetilde{K}$ with $|n|>1$ then:

\[|\widetilde{\theta}(\beta)|>\frac{5\pi}{2}.\]
 Applying Lemma \ref{angle} we obtain:

\[ |\widetilde{\theta}(\alpha_{0} \alpha_{1})|< |\widetilde{\theta}(\alpha_{0})|+|\widetilde{\theta}(\alpha_{1})| + \frac{\pi}{2} < \frac{3}{2} \pi \]
We conclude that $n=\pm1$. Now:
\begin{itemize}
\item If $n=1$ we set:\[\gamma(\alpha_{0} \alpha_{1})\gamma^{-1}=(\gamma \alpha_{0} \gamma^{-1})(\gamma \alpha_{1} \gamma^{-1}) \in \widetilde{K}\widetilde{K}.\]
\item If $n=-1$ then $\gamma(\alpha_{0} \alpha_{1})\gamma^{-1} \in -\pi \widetilde{K}$ and $\gamma(\alpha_{0} \alpha_{1})^{-1} \gamma^{-1} \in \pi \widetilde{K}$. We then write:
\[\gamma(\alpha_{0} \alpha_{1})^{-1}\gamma^{-1}=(\gamma \alpha_{1}^{-1} \gamma^{-1})(\gamma \alpha_{0}^{-1}\gamma^{-1}) \in \widetilde{K} \widetilde{K}.\]
\end{itemize}
\end{proof}

\begin{lemma}\label{lema_3}
	 For every $\alpha \in \pi \widetilde{K}$ there are elements $\beta_1 \beta_{2} \in \widetilde{G}$ such that $\alpha=\beta_{1}\beta_{2}\beta_{1}^{-1}\beta_{2}^{-1}$.
\end{lemma}
\begin{proof}
	By the previous lemma we have $\beta_{1},\beta_{3} \in \widetilde{K}$ such that $\alpha=\beta_{1}\beta_{3}$. Since $\beta_{1} ^{-1} \in \widetilde{K}$, there exists $\beta_{2} \in \widetilde{G}$ so that $\beta_{2} \beta_{1}^{-1} \beta_{2}^{-1}=\beta_{3}$.\\
	Then $$\alpha=\beta_{1} \beta_{3}=\beta_{1} \beta_{2} \beta_{1}^{-1}\beta_{2}^{-1}$$
\end{proof}
\begin{lemma}\label{list}
	For every $n \geq1 $ there are $\gamma_1,\dots, \gamma_{n+1} \in \widetilde{K}$ so that $\gamma_1 \dots \gamma_{n+1}=n \pi\alpha_{0}.$
\end{lemma}
\begin{proof}
	We will argue by induction. The case $n=1$ is Lemma \ref{productmil}. Now assume that the lemma holds for $n$ so that we can write:
	\[\gamma_1 \dots \gamma_{n+1}=n \pi\alpha_{0}.\] Then:
	
	\[ (n+1)\pi \alpha_0= \pi (n \pi) \alpha_0=\pi (\gamma_1 \dots \gamma_{n+1})=(\pi \gamma_1) \dots \gamma_{n+1}=\gamma \gamma' \gamma_2 \dots \gamma_{n+1}.\]
	Here $\gamma, \gamma' \in \widetilde{K}$ and we have applied the Lemma \ref{productmil} to $\pi \gamma_1$.

\end{proof}

\begin{theorem}\label{Milnor}
A rank two oriented vector bundle $\pi:E \rightarrow \Sigma$ admits a flat connection if and only if
\[ |D(E)|<g.\]
\end{theorem}

\begin{proof}
One direction is guaranteed by Lemma \ref{bound}. We want to prove that a bundle $E$ with $|D(E)|<g$ admits a flat connection. It suffices to show that
there is some representation $\rho: \pi_1(\Sigma) \rightarrow G$ such that $\delta(\rho)=D(E)$. Since the existence of a flat connection does not depend on the orientation we may assume that $D(E)>0$. Furthermore, some of the matrices in the representation can be set to equal the identity, therefore it suffices to show the case where $D(E)=g-1$ and $g>1$.
By Lemma \ref{list} there exist $\gamma_{1},...,\gamma_{g-1} \in \widetilde{K}$ so that:
\[\gamma_1 \dots \gamma_{g-1}=(g-2) \pi\alpha_{0}.\]
If we set $\gamma_g=\alpha_{0}^{-1}$ then:
\[\gamma_1 \dots \gamma_{g-1}\gamma_g=(g-2) \pi.\]
Now by Lemma \ref{lema_3}, for every $1 \leqslant i \leqslant g$ there exists $\alpha_i,\beta_i \in \widetilde{G}$ so that
\[[\alpha_i,\beta_i]= \pi \gamma_i.\]
Then
\[\prod_{i=1}^g[\alpha_i,\beta_i]=2\pi (g-1).\]
This implies that by setting $A_i=p(\alpha_i)$ and $B_i=p(\beta_i)$ one obtains a representation $\rho$ such that:
\[\delta(\rho)=\frac{1}{2\pi}\prod_{i=1}^g[\alpha_i,\beta_i]= g-1.\]
\end{proof}

\begin{corollary}
The torus is the only closed oriented surface that admits an affine structure. In particular, Chern's conjecture holds in dimension $d=2$.
\end{corollary}
\begin{proof}
Since $S^2$ is simply connected, a flat connection on $TM$ would give a trivialization of $TM$, which does not exist because the Euler characteristic of the sphere is $2$.
Let us now assume that $g>1$. Then $D(TM)=\chi(\Sigma)=2-2g$. Thererfore:
\[|D(TM)|= 2|(1-g)|=2g-2 \geq g.\]
\end{proof}
\subsection{The counterexamples of Smillie}

Chern's conjecture asks about the Euler characteristic of closed affine manifolds. For a while, it was not known whether the condition that the connection is torsion free was essential. A manifold $M$ is said to be flat if its tangent bundle admits a flat (not necessarily torsion free) connection. A strong version of Chern's conjecture asking whether the Euler characteristic of a closed flat manifold vanishes was an open problem for a while. This question was answered by Smillie \cite{Smillie}, who proved the following

\begin{theoremnn}[Smillie]
For each $n>1$ there are closed flat manifolds $M^{2n}$ which have non-zero Euler characteristic.
\end{theoremnn}

In what follows we will present Smillie's construction.
\begin{definition}
Abusing the notation, the trivial vector bundle of rank $n$ over $M$ will be denoted $\R^n$. It will be clear from the context wether we are referring to the trivial bundle or the euclidean space. We will say that a vector bundle $\pi:E \rightarrow M$ is almost trivial if \[E \oplus \R\simeq \R^m.\]
We will say that $M$ is almost parallelizable if $TM$ is almost trivial.
\end{definition}

\begin{lemma}\label{pull}
An oriented vector bundle $\pi: E \rightarrow M$ is almost parallelizable if and only if there exists a function $f: M \rightarrow S^m$ such that:
\[ f^*(TS^m)\simeq E.\]
\end{lemma}

\begin{proof} Since $TS^m$ is almost parallelizable, so is $f^*(TS^m)$.
On the other hand, suppose that $E$ is almost parallelizable. Once we fix an isomorphism
\[ \varphi: E \oplus \R \longrightarrow \R^{m+1},\]
there is a unique smooth function $f: M \rightarrow S^m$ such that:
\begin{itemize}
\item $f(p) \in (\varphi(E_p))^\bot$.
\item If $\{e_1, \dots , e_m\}$ is an oriented basis for $E_p$ then \[\{ f(p), \varphi(e_1),\dots, \varphi(e_m)\}\] is
an oriented basis for $\R^{m+1}$.
\end{itemize}
By construction $\varphi$ restricts to an isomorphism from $E$ to $f^*(TS^m)$.
\end{proof}

\begin{exercise}\label{emb}
Show that there exists an embedding:
\[ \iota:S^m \times S^n \rightarrow \R^{m+n+1}\]
with trivial normal bundle. Conclude that $S^m \times S^n$ is almost parallelizable.
\end{exercise}
\begin{lemma}\label{product}
Let $\pi_1: E \rightarrow M$ and $\pi_2: E' \rightarrow N$ be almost trivial vector bundles. Then the vector bundle ${p_1}^*E \oplus {p_2}^*E'$ over $M \times N$ is almost trivial. \end{lemma}
\begin{proof}
By Lemma \ref{pull}, there are functions $f: M \rightarrow S^m$ and $g: N \rightarrow S^n$ such that
\[ E\simeq f^*(TS^m); \text{ and } E'\simeq g^*(TS^n).\]
By Exercise \ref{emb}, there is a function $h: S^m \times S^n \rightarrow S^{m+n}$ such that
\[ h^*(TS^{m+n})=T(S^m \times S^n).\]
Then we have
\[ (h \circ (f\times g))^*(TS^{m+n})=(f\times g)^* (h^*(TS^{m+n}))=(f \times g)^*(T(S^m \times S^n))={p_1}^*E \oplus {p_2}^*E.\]
By Lemma \ref{pull}, we conclude that ${p_1}^*E \oplus {p_2}^*E$ is almost trivial.
\end{proof}

\begin{definition}
Let $M,N$ be closed oriented manifolds of dimension $m$ and $f: M \rightarrow N$ be a smooth function.
The degree of $f$ is the number:
\[ \mathsf{deg}(f):= \int_M f^*(\omega)\]
where $\omega$ is any $m$-form on $N$ such that:
\[ \int_{N} \omega=1.\]
\end{definition}

\begin{exercise}
Show that the number $\mathsf{deg}(f)$ defined above is an integer that does not depend on the choice of $\omega$.
Show also that if $f$ and $g$ are smoothly homotopic then they have the same degree. \end{exercise}

It turns out that the degree is a complete invariant of maps from a closed oriented $m$-manifold to $S^m$. This is a theorem of Heinz Hopf, whose proof can be found in \cite{BT}.

\begin{theorem} [Hopf Degree Theorem] Let $M$ be a closed, connected, oriented
$m$-dimensional manifold and $f,g:M \rightarrow S^m$ be smooth maps.
Then $f$ and $g$ are smoothly homotopic if and only if they have the same degree.
\end{theorem}

\begin{lemma}\label{Euler}
Let $M$ be a closed oriented manifold of dimension $m$ and $E,E'$ oriented almost trivial vector bundles of rank $m$.
Then $E$ and $E'$ are isomorphic if and only if they have the same Euler class.
\end{lemma}
\begin{proof}
Clearly, if the bundles are isomorphic they have the same Euler class. Let us prove the converse. By Lemma \ref{pull},
we know that there are functions $f,g:M \rightarrow S^m$ such that
\[ f^*(TS^m) \simeq E; \text{ and } g^*(TS^m)\simeq E'.\]
Since $E$ and $E'$ have the same Euler class we conclude that $f$ and $g$ have the same degree. The Hopf degree theorem implies that $f$ and $g$ are smoothly homotopic and therefore $E $ is isomorphic to $E'$.
\end{proof}

\begin{lemma}\label{connected}
The connected sum of almost parallelizable manifolds is almost parallelizable.
\end{lemma}
\begin{proof}
Let $f:M \rightarrow S^m$ and $g: N \rightarrow S^m$ be smooth maps such that
\[ f^*(TS^m) \simeq TM \text{ and } g^*(TS^m)\simeq TN.\]
Fix $p \in M$  and $q \in N$ and coordinates $\varphi: U \rightarrow \R^m$ and $\phi: W \rightarrow \R^m$ around
$p$ and $q$ respectively. Set $X=\varphi^{-1}(B(0,1))\subseteq U$ and $Y= \phi^{-1}(B(0,1))\subseteq W$. Since $\pi_{m-1}(S^m)=0$ we may assume that
\[ (f \circ \varphi^{-1})\vert_{S^{m-1}}=(g \circ \phi^{-1}) \vert_{S^{m-1}}.\]
Then there is a well defined map
\[ h: M\#N \rightarrow S^m\]
such that $h\vert_{M-X}=f$ and $h\vert_{N-Y}=g$. We conclude that \[ h^*(TS^m)\simeq T(M \#N)\]
and therefore $M \#N$ is almost parallelizable.
\end{proof}

\begin{theorem}[Smillie]
Let $\Sigma_g$ be the closed oriented surface of genus $g$ and $P= S^1 \times S^3$. Then
\[M^4= (\Sigma_3 \times \Sigma_3)\# \underbrace{ P\# \cdots \#P}_{6\text{ times}}\]
and
\[M^6=((\Sigma_3 \times \Sigma_3)\# \underbrace{ P\# \cdots \#P}_{9\text{ times}})\times \Sigma_3\]
are closed flat manifolds with non-zero Euler characteristic. By taking products of $M^4$ and $M^6$ one obtains
closed flat manifolds with nonzero Euler characteristic in all even dimensions greater than $d=2$.
\end{theorem}

\begin{proof}
The Euler characteristic of a connected sum of closed even dimensional manifolds satisfies:
\[ \chi( M \#N)=\chi(M)+ \chi(N)-2.\]
Using this formula, we compute
\[ \chi(M^4) =\chi(\Sigma_3 \times \Sigma_3)+\chi( \underbrace{ P\# \cdots \#P}_{6\text{ times}})-2=16-10-2=4.\]
and
\[ \chi(M^6) =(\chi(\Sigma_3 \times \Sigma_3)+\chi( \underbrace{ P\# \cdots \#P}_{9\text{ times}})-2)\times \chi(\Sigma_3)=(16-16-2)\times (-4)=8.\]
It remains to prove that $M^4$ and $M^6$ are flat manifolds. Lemmas \ref{connected} and \ref{product} imply that $M^4$ and $M^6$ are almost parallelizable. Let $h: \Sigma_3 \rightarrow S^2$ be a degree one map and set $E=h^*(TS^2)$. Then $E$ is almost parallelizable and
\[ \int_{\Sigma_3}e(E)=\int_{S^2} e(TS^2)=2.\]
Lemma \ref{product} implies that  ${\pi}^*E \oplus {\pi}^*E$ is an almost trivial bundle over $\Sigma_3 \times \Sigma_3.$ Let \[f: M^4 \rightarrow \Sigma_3 \times \Sigma_3\] be the map that
sends \[\underbrace{ P\# \cdots \#P}_{6\text{ times}}\] to a point. Then $f$ has degree $1$ and therefore:
\[ \int_{M^4} e(f^*(E \oplus E))= \int_{M^4} f^*(e(\pi(E \otimes E))=\int_{\Sigma_3 \times \Sigma_3} {\pi}^*(e(E))\wedge
{\pi}^*(e(E))=\Big(\int_{\Sigma_3}e(E)\Big)^2=4.\]
We conclude that:
\[ e(TM^4)=e({\pi}^*E \oplus {\pi}^*(E))\]
and thus, by Lemma \ref{Euler}
\[ TM^4 \simeq {\pi_1}^*E \oplus {\pi_2}^*(E).\]
Milnor's inequality stated in Theorem \ref{Milnor} implies that $E$ admits a flat connection and therefore, so does $TM^4$.
Similarly, Let \[g:\Sigma_3 \times \Sigma_3\# \underbrace{ P\# \cdots \#P}_{9\text{ times}}  \rightarrow \Sigma_3 \times \Sigma_3\] be the map that
sends \[\underbrace{ P\# \cdots \#P}_{9\text{ times}}\] to a point. Then $g$ has degree one and therefore the map
\[ z:=g \times \mathsf{id}:M^6 \rightarrow \Sigma_3 \times \Sigma_3 \times \Sigma_3\]
also has degree one. This implies that
\[ \int_{M^6} e(z^*({\pi}^*E \oplus {\pi}^*E\oplus {\pi}^*E))= \Big(\int_{\Sigma_3} e(E) \Big)^3=8.\]
We conclude that
\[ e(TM^6)=e({\pi}^*E \oplus {\pi}^*E \oplus {\pi}^*E)\]
and therefore
\[ TM^6 \simeq {\pi}^*E \oplus {\pi}^*E \oplus {\pi}^*E .\]
Theorem \ref{Milnor} implies that $E$ admits a flat connection and therefore, so does $TM^6$.

\end{proof}

\subsection{The case of complete manifolds, after Sullivan and Kostant}

Here we will prove the following theorem of Kostant and Sullivan which solves Chern's conjecture in the case of complete affine manifolds.

\begin{theoremnn}[Kostant-Sullivan]\label{Theocomplete}
	If $M$ is a closed affine complete manifold then $\chi(M)=0$.	
\end{theoremnn}

The idea of their proof is that even though the connection on $M$ may not admit a compatible metric, the Chern-Weil theory of characteristic classes can still be used to prove that the Euler class vanishes.

\begin{definition}
	We will say that a subgroup $G$ of $\mathrm{GL}(m,\R)$ is $1$-spectral if for every $ A \in G$, \[\mathsf{det}(A-\mathsf{id})=0.\]
	\end{definition}
	
\begin{definition}
	We will say that a Lie subalgebra $\mathfrak{g}$ of $\mathfrak{gl}(m,\R)$ is singular if for all $v \in \mathfrak{g}$,
	\[ \mathsf{det(v)=0}.\]
\end{definition}

The key observation of Sullivan and Kostant is that if the frame bundle of $TM$ admits a reduction to a connected compact and $1$-spectral subgroup of $\mathrm{GL}^+(m,\R)$ then the Euler class of $TM$ vanishes.

\begin{lemma}
	If $G$ is a $1$-spectral Lie subgroup of $\mathrm{GL}(m,\R)$ then its Lie algebra $\mathfrak{g}$ is singular.
\end{lemma}
\begin{proof}
	It is enough to prove that there is an open subset  $W\subseteq \mathfrak{g}$ such that $0\in W$  and all matrices of $W$ are singular.
	Choose $W$ such that the exponential map $\mathsf{Exp}:W \to U\subseteq G$ is a diffeomorphism with inverse $\mathsf{Log}: U\to W$. For $v\in W$, $A\in U$ these functions are given explicitly by
	$$\mathsf{Exp}(v)=\sum_{k\geqslant0} \frac{x^{k}}{k!},$$
	$$\mathsf{Log}(A)=\sum_{k\geqslant 1} \frac{(-1)^{k+1}(A-\mathsf{id})^{k}}{k}.$$
	Any $v \in W$ is of the form $v=\mathsf{Log}(A)$ with $\mathsf{det}(A-\mathsf{id})=0.$ Thus there exists a vector $x \in \R^m$ such that $Ax=x$. Then
	\[v(x):=\sum_{k\geqslant 1} \frac{(-1)^{k+1}(A-\mathsf{id})^{k}}{k} (x)=\sum_{k\geqslant 1} \frac{(-1)^{k+1}(A-\mathsf{id})^{k}(x)}{k}=0.\]
	We conclude that $v$ is singular.
\end{proof}

\begin{lemma}\label{compact}
	Let $M$ be a manifold of dimension $m$ and $P$  a $\mathrm{GL}^{+}(m,\R)$ principal bundle over $M$  that admits a reduction to a subgroup $G$ which is compact, connected and $1$-spectral. Then the Euler class of $P$ vanishes.
\end{lemma}

\begin{proof}
The Chern-Weil homomorphism \[\mu_P: H(B\mathrm{GL}^+(m,\R))\rightarrow H(M)\] factors through the restriction map
\[\iota^*: H(B\mathrm{GL}^+(m,\R)) \rightarrow H(BG).\] Moreover, there is a commutative diagram

\[
\xymatrix{
H(B\mathrm{GL}^+(m,\R))\ar[r]^-{\iota^*} \ar[d]^-\simeq &H(BG) \ar[d]^-{\simeq}\\
   S(\mathfrak{gl}(m,\R)^*)^{\mathrm{GL}^+(m,\R)} \ar[r]^-{\rho}& S(\mathfrak{g}^*)^G
}
 \]
 where the vertical arrows are isomorphisms and $\rho$ is the restriction of polynomials. Since the Euler class corresponds to the Pfaffian polynomial, it suffices to prove that $\rho(\mathsf{Pf})=0$. For this we observe that since $\mathfrak{g}$ is singular
	\[[\rho(\mathsf{Pf})(v)]^2=\mathsf{det}(v)^2=0.\]
	We conclude that $\rho(\mathsf{Pf})=0$ and therefore the Euler class of $P$ vanishes.
	\end{proof}

\begin{lemma}\label{s2}
	Let $M$ be a manifold of dimension $m$ and $P$ a $\mathrm{GL}^{+}(m,\R)$ principal bundle over $M$ that admits a reduction to a subgroup $G$ which is closed, connected and $1$-spectral. Then the Euler class of $P$ vanishes.
\end{lemma}
\begin{proof}
Since $G$ is a closed subgroup of the general lineal group, it is a Lie subgroup.
	Let $K\subseteq G$ be a maximal compact subgroup. Since $G/K$ is contractible, the natural map $\pi: BK \rightarrow BG$ is a homotopy equivalence.
	The Chern-Weil homomorphism \[\mu_P: H(B\mathrm{GL}^+(m,\R))\rightarrow M\] factors through the restriction map
\[ \iota^*: H(B{\mathrm{GL}}^+(m,\R)) \rightarrow H(BG)\] so it suffices to show that $\iota^*$ sends the Euler class to zero. There is also a commutative diagram

\[\xymatrix{
 H(B\mathrm{GL}^+(m,\R))\ar[r]^-{\pi^* \circ \iota^*} \ar[d]^{\simeq} &H(BK) \ar[d]^-{\simeq}\\
   S(\mathfrak{gl}(m,\R)^*)^{\mathrm{GL}^+(m,\R)} \ar[r]^-{\rho}& S(\frak{K}^*)^K
	}\]	
By the Lemma \ref{compact}, we know that ${\rho}(\mathsf{Pf})=0$ and therefore $ \pi^* \circ \iota^*$ sends the Euler class to zero. Since $\pi^*$ is an isomorphism, we conclude that $\iota^*$ sends the Euler class to zero.	
\end{proof}

\begin{lemma}\label{lemmaspectral}
	Let $M$ be a manifold of dimension $m$ and $P$ a $\mathrm{GL}^{+}(m,\R)$ principal bundle over $M$ that admits a reduction to  a subgroup $G \subseteq \mathrm{GL}^{+}(m,\R)$ which is 1-spectral and closed. If $G$ has finite number of  connected components  then the Euler class of $P$ is zero.
\end{lemma}
\begin{proof}
As in the previous lemmas, it suffices to show that the restriction map:
\[ \iota^*: H(B\mathrm{GL}^+(m,\R)) \rightarrow H(BG)\]
sends the Euler class to zero. Let $G_0$ be the connected component of the identity in $G$. Then the natural map
\[ \pi: BG_0 \rightarrow BG\]
is a finite covering and therefore it induces an injective map in cohomology. Therefore it suffices to show that the map
\[ \pi^* \circ \iota^*: H(B\mathrm{GL}^+(m,\R)) \rightarrow H(BG_0)\]
sends the Euler class to zero. This is guaranteed by Lemma \ref{s2}.

\end{proof}

In view of the lemma above, we are left with the problem of showing that the frame bundle of a closed complete affine manifold admits a reduction to a closed 1-spectral group with finitely many connected components.

\begin{lemma}\label{z}
Let $G$ be a 1 spectral subgroup of $\mathrm{GL}(m,\R)$ and $\overline{G}$ its Zariski closure.
	Then:
	\begin{enumerate}
		\item  $\overline{G}$ is a subgroup of $\mathrm{GL}(m,\R)$.
		\item $\overline{G}$ is a closed Lie group in $\mathrm{GL}(m,\R)$.
		\item $\overline{G}$ is 1-spectral.
		\item $\overline{G}$ has finitely many connected components.
	\end{enumerate}
\end{lemma}

\begin{proof}
	 Since the multiplication and inverse functions are algebraic operations, they are continuous in the Zariski topology.
	 Fix $x\in G$ and consider the homeomorphism 	
	 \[\begin{array}{rcc}
	L_x: \mathrm{GL}(m,\R) & \longrightarrow & \mathrm{GL}(m,\R) \\
	 y & \longmapsto & xy
	 \end{array}.\\ \]
	 Then \[x \overline{G}= \overline{x G}\subseteq \overline{G}.\]
	 Fix now $y \in \overline{G}$ and consider the homeomorphism $R_y$ given by right multiplication by $y$. Then
	 \[  \overline{G}y=\overline{ G y}\subseteq \overline{G}.\]
	 We conclude that $\overline{G}$ is closed with respect to the product. The map $x \mapsto x^{-1}$ is a homeomorphism
	 of $\mathrm{GL}(m,\R)$ and therefore:
	 \[ \overline{G}^{-1}=\overline{G^{-1}}=\overline{G}.\]
	 We conclude that $G$ is a group.
	 The second statement is true because $\overline{G}$ is closed in the Zariski topology and therefore also in the smooth topology. The third statement holds because the determinant is a continuous function. The last statement is true because any real algebraic set has finitely many connected components. 	
	
 \end{proof}

We can now prove Chern's conjecture for complete manifolds:
\begin{theorem}[Kostant-Sullivan]\label{Theocomplete}
	If $M$ is a closed affine complete manifold then $\chi(M)=0$.
\end{theorem}

\begin{proof}
	In view of Lemma \ref{lemmaspectral}, it suffices to show that the frame bundle of $TM$ admits a reduction to
	a closed 1-spectral subgroup of $\mathrm{GL}^+(m,\R)$ which has finitely many connected components. By Proposition \ref{develop}, we know that $M$ is the quotient $ \R^m / \Gamma$ where $\Gamma \subset \mathsf{Aff}(\R^m)$ is isomorphic to the fundamental group of $M$. Consider the natural homomorphism
	\[ \lambda:  \mathsf{Aff}(\R^m) \rightarrow \mathrm{GL}(m,\R), \, Ax+b \mapsto A \]
	and let $G$ be the group $\lambda(\Gamma)$. We claim that $G$ is 1-spectral. Take $g=Ax+b \in \mathsf{Aff}(\R^m)$. If $A$ is not 1-spectral then the equation $Ax+b=x$ has a solution, which is imposible because $\Gamma$ acts freely on $\R^m$.  Lemma \ref{z} guarantees that $\overline{G}$ is a closed 1-spectral group with finitely many connected components. Therefore, it suffices to prove that the frame bundle of $TM$ admits a reduction to the group $G$ and therefore to $\overline{G}$. Consider the projection \[ \pi: \R^m \rightarrow M\] and for each $p\in M$, the following subset of the frame bundle of $T_pM$
	 \[ S_p:=\{ \phi: \R^m \rightarrow T_pM: \phi \text{ is a linear isomorphism of the form } \phi= D\pi(x) \text{ for } x \in \pi^{-1}(p)\} \]
The group $G$ acts on the right by composition and this action is free and transitive. If we set:
	 \[S:=\coprod_p S_p\subset \mathsf{Fr}(TM)\]
	 we obtain a reduction of the structure group of the frame bundle of $TM$ to $G$.
	\end{proof}

\section{The case of special affine manifolds, after Klingler}
In the case of special affine manifolds, Chern's conjecture holds.
\begin{definition}
An affine manifold $M$ is called special if it admits a covariantly constant volume form $\omega \in \Omega^m(M)$.
\end{definition}

\begin{lemma}
An affine manifold is special if and only if there exists an affine atlas for which the transition functions are of the form:
\[ f(x)=Ax+b\]
where $\mathsf{det}(A)=1$. We will call such an atlas a special atlas.
\end{lemma}
\begin{proof}
Suppose $\omega \in \Omega^m(M)$ is a parallel volume form. Then the family of affine coordinates $\varphi: U \rightarrow V \subseteq \R^m$ such that:
\[ dx_1 \wedge \dots \wedge dx_m =\omega\vert_{U}\]
is a special atlas. Conversely, given a special atlas $\{ U_\alpha, \varphi_\alpha\}$ there is unique parallel volume form $\omega \in \Omega^m(M)$ such that for special affine coordinates:
\[ dx_1 \wedge \dots \wedge dx_m =\omega\vert_{U_\alpha}\]
\end{proof}

In what follows we will present Klingler's proof \cite{K} of the following remarkable theorem:
\begin{theoremnn}[Klingler]
The Euler characteristic of a closed special affine manifold is zero.
\end{theoremnn}
The structure of Klingler's proof is as follows. Proposition \ref{para} shows that if $M$ is an affine manifold then
$TM$ has a canonical para-hypercomplex structure and therefore there are spectral sequences $\mathcal{B}^{ij}_r$ and
$\mathcal{B}^{ij}_{c,r}$ converging to $H(TM)$ and $H_c(TM)$ respectively. Proposition \ref{spectralt} shows that if the natural map induced by the inclusion:
\[ \lambda: \mathcal{B}^{0m}_{c,\infty}\rightarrow \mathcal{B}^{0m}_\infty\]
is zero then the Euler characteristic of $M$ is zero. Finally, Proposition \ref{vanishingt} shows that if $M$ is special affine then:
\[ \lambda: \mathcal{B}^{0m}_{c,\infty}\rightarrow \mathcal{B}^{0m}_\infty\]
is the zero map.

\subsection{Local product structures and affine manifolds}
\begin{definition}
  A local product structure on $M$ is a pair of foliations $H,V \subset TM$ such that
  \[TM=H \oplus V.\]
  Vector fields tangent to $H$ will be called horizontal and those tangent to $V$ will be called vertical.
\end{definition}

Of course, if $M = X \times Y$ then $M$ has a local product structure given by: \[H=\mathsf{Ker}(D\pi_2)\quad \text{and}\quad V =\mathsf{Ker}(D\pi_1).\] Thus, a product structure induces a local product structure.\\

Let $M$ be a manifold with a local product structure  $H\oplus V=TM$. Then there is a natural decomposition
\[\Omega(M)=\Gamma(\Lambda(H\oplus V)^*)=\Gamma\left(\Lambda H^* \otimes \Lambda V^*\right).\]
We set
\[
\Omega^{ij}(M):=\Gamma\left(\Lambda^i H^*\otimes \Lambda^j V^*\right),
\]
so that
\[ \Omega^k(M)=\bigoplus_{i+j=k}\Omega^{ij}(M).\]
A form $\eta \in \Omega^{ij}(M)$ vanishes when evaluated on more than $i$ horizontal vector fields or more than $j$ vertical vector fields.

\begin{lemma}
Let $M$ be a manifold with a local product structure $TM=H\oplus V$. Then the de-Rham operator
\[ d: \Omega(M) \rightarrow \Omega(M)\]
decomposes as a sum
\[d=d_H+d_V\]
where
\[ d_H:\Omega^{ij}(M) \rightarrow \Omega^{i+1,j}\quad \text{and}\quad d_V:\Omega^{ij}\rightarrow \Omega^{i,j+1}(M).\]
Thus, $\Omega^{ij}(M)$ becomes a double complex.
\end{lemma}

\begin{proof}
Recall that for a differential form $\eta \in \Omega^k(M)$ the de-Rham operator is given by the formula
\begin{align*}
d\eta(X_0,\cdots,X_k)=&\sum_{i=0}^k(-1)^iX_i(\eta(X_0,\dots,\hat X_i,\cdots,X_k))\nonumber\\
&+\sum_{r<s}(-1)^{r+s}\eta([X_r,X_s],X_0,\dots,\hat X_r,\dots, \hat X_s,\dots,X_k).
\end{align*}
 Consider $\eta \in \Omega^{ij}(M)$ with $k=i+j$. We claim that \[d\eta(X_0, \dots ,X_{k})=0\] if there are more than
 $i+1$ horizontal vector fields. Suppose this is the case. Then the list
 $X_0,\dots,\hat X_i,\dots,X_k$ contains at least $i+1$ horizontal vector fields and therefore
\[ \eta(X_0,\cdots,\hat X_i,\cdots,X_k)=0.\]
Also, since $H$ is a foliation, the list
\[[X_r,X_s],X_0,\cdots,\hat X_r,\dots, \hat X_s,\dots,X_k\]
contains at least $i+1$ horizontal vector fields and therefore
\[\eta([X_r,X_s],X_0,\dots,\hat X_r,\dots, \hat X_s,\cdots,X_k)=0.\]
By symmetry, it is also true that $d\eta(X_0,\dots,X_k)=0$ of there are more than $j+1$ vertical vector fields.
Thus, we conclude that \[d\eta \in \Omega^{i+1,j}(M) \oplus \Omega^{i, j+1}(M),\] as claimed.

\end{proof}

We have seen that the de-Rham complex of a manifold with a local product structure decomposes as a double complex.
This induces two decreasing filtrations which we will denote as follows:
\[
F^V_r(\Omega(M))=\bigoplus_{ j\geq r}\Omega^{ij}(M);\]\[F^H_r(\Omega(M))=\bigoplus_{ i\geq r}\Omega^{ij}(M).
\]
Notice that forms in $F^V_r(\Omega^k(M))$ are precisely the $k$-forms that vanish when evaluated in $k+1-r$ horizontal vector fields, hence this filtration is just the filtration associated to the horizontal foliation $H$.
We will be concerned mostly with a special kind of local product structures called paracomplex structures, which we now introduce.

\begin{definition}
For any endomorphism of the tangent bundle $A \in \mathsf{End}(TM)$ the Nijenhuis tensor of $A$ is defined by
\[
N_A(X,Y)=-A^2[X,Y]+A([AX,Y]+[X,AY])-[AX,AY].
\]
\end{definition}

\begin{exercise}
Show that the Nijenhuis tensor is indeed a tensor.
\end{exercise}

\begin{definition}
An almost-complex structure on $M$ is an endomorphism $I\in  \mathsf{End}(TM)$ such that $I^2=-\id$. A complex structure on $M$ is an almost complex structure $I$ such that $N_I=0$. A para-complex structure on $M$ is an endomorphism $J\in \mathsf{End}(TM)$ such that $J^2=\id$ and such that $TM=H\oplus V$ where $H$ and $V$ are the eigenspaces of $1$ and $-1$ respectively, and are integrable distributions of the same dimension. A para-hypercomplex structure on $M$ is the data of a complex structure $I$ and a para-complex structure $J$ satisfying $IJ=-JI$.
\end{definition}

An almost complex structure on $M$ gives $TM$ the structure of a complex vector bundle by:
\[ z X= aX + bI(X)\]
where $z =a+bi \in \MC$.
\begin{lemma}
 Let $M$ be a manifold with a para-hypercomplex structure. Then for any $z=a+bi\in S^1\subseteq\MC$, the endomorphism $J_z \in \mathsf{End}(TM)$ given by \[J_z(X):=z J(X).\] is
 a para-complex structure.
 \end{lemma}
 \begin{proof}
 First we observe that \[J_z^2=zJ(zJ)=z\bar z J^2=\id.\]
Let $H$ and $V$ denote the eigenspaces of $J$ associated to 1 and $-1$ respectively. Then the eigenspaces of $J_z$ are $(1+z)H$ and $(1+z)V$. Let us check that these distributions are involutive. Take $X,Y\in H$ and write $z=a+bi$.  Then
\begin{align*}
  [(1+z)X,(1+z)Y]=&(1+a)[X,Y]+b\big([IX,Y]+[X,IY]\big)+\\
  &a\big((1+a)[X,Y]+b\big([IX,Y]+[X,IY]\big)\big)+\\
  &b^2[IX,IY]
\end{align*}
Since we have $N_I=0$, we may replace $[IX,Y]+[X,IY]=i[X,Y]-I[IX,IY]$ in the previous expression, which yields
\begin{align*}
  [(1+z)X,(1+z)Y]=&(1+a+bI)[X,Y]+a(1+a+bI)[X,Y]+\\
  &-bI[IX,IY]-abI[IX,IY]+b^2[IX,IY]
\end{align*}
Ultimately, we get
\[
[(1+z)X,(1+z)Y]=(1+z)[X,Y]+a(1+z)[X,Y]+-b(1+z)I[IX,IY].
\]
The anticommutativity relation $JI=-IJ$ implies that $I$ interchanges the eigenspaces $H$ and $V$. In particular we have that $I[IX,IY]\in H$. Therefore $[(1+z)X,(1+z)Y]=(1+z)Z$ with $Z\in H$. A similar argument shows that $(1+z)V$ is also involutive.
\end{proof}

\begin{proposition}\label{para}
Let $M$ be an affine manifold. Then $TM$ has a canonical para-hypercomplex structure.
\end{proposition}

\begin{proof}
Fix affine coordinates $x_1, \dots, x_m$ in $M$ which determine coordinates $x_1,\dots x_m, y_1,\dots,y_m$ in $TM$ and set
\[ J\Big(\frac{\partial}{\partial x_i}\Big)=\frac{\partial}{\partial x_i},\,\, J\Big(\frac{\partial}{\partial y_i}\Big)=-\frac{\partial}{\partial y_i},\]
and
\[ I\Big(\frac{\partial}{\partial x_i}\Big)=\frac{\partial}{\partial y_i}, \,\, I\Big(\frac{\partial}{\partial y_i}\Big)=-\frac{\partial}{\partial x_i}.\]
It is clear that $IJ+JI=0$, $I^2=-\id$ and the eigenspaces of $J$ are involutive. Evaluating $N_I$ in any couple of generators leads to a sum of terms of the form $\left[\frac{\partial}{\partial x_i},\frac{\partial}{\partial x_j}\right]$, $\left[\frac{\partial}{\partial y_i},\frac{\partial}{\partial y_j}\right]$ or $\left[\frac{\partial}{\partial x_i},\frac{\partial}{\partial y_j}\right]$, all of which vanish. Finally, since the transition functions are affine, the operators $I$ and $J$ are well defined.
\end{proof}

\subsection{The spectral sequence and the Euler class}

As described in the previous section, an affine structure on a manifold $M$ induces a para-hypercomplex structure on its tangent bundle $TM$.  Klingler's proof relies heavily on the spectral sequences associated to this para-hypercomplex structure.\\

Given a manifold $M$ we will denote by:
\[ \lambda: \Omega_c(M) \rightarrow \Omega(M)\]
the inclusion of differential forms with compact support into all forms. The vanishing of the Euler characteristic can be reformulated as follows.

\begin{lemma}\label{in}
  Let $M$ be a connected closed manifold of dimension $m$ and $E$ an oriented vector bundle over $M$ of rank $r$. The Euler class $e(E)$ is zero if and only if the natural morphism
  \[
  \lambda:H_c^r(E,\R)\to H^r(E,\R)
  \]
  vanishes. Here $H_c(E, \R)$ denotes the cohomology with compact support of $E$.
\end{lemma}
\begin{proof}
Let $\tau \in H_c(E, \R)$ be the Thom class of $E$. By the Thom isomorphism
\[ H^r_c(E, \R) \simeq H^0(M, \R)\simeq \R\]
and $\tau$ generates $H_c(E,\R)$. If $\sigma \in \Gamma(E)$ is the zero section then
\[ e(E)= \sigma^* (\lambda(\tau)).\]
Since \[\sigma^*:H(E) \rightarrow H(M)\] is an isomorphism, we conclude that the map $\lambda$ is zero if and only if $e(E)=0$.
\end{proof}

Let $M$ be a closed affine manifold. We know that $TM$ is a para-hypercomplex manifold.
\begin{itemize}
\item We will denote by $\mathcal{B}^{pq}_r$ the spectral sequence associated with the filtration $F^V(\Omega(TM))$.
\item We will denote by $\mathcal{B}^{pq}_{c,r}$ the spectral sequence associated with the filtration $F^V(\Omega_c(TM))$.
\end{itemize}

\begin{proposition}\label{spectralt}
Let $M$ be a closed connected affine manifold of dimension $m$. If the morphism
\[ \lambda: \mathcal{B}^{0m}_{c, \infty}\rightarrow \mathcal{B}^{0m}_\infty\]
is zero then the Euler characteristic of $M$ is zero.
\end{proposition}
\begin{proof}
By the Thom isomorphism we know that $H_c^m(TM)$ is a one dimensional vector space and therefore
\[ \bigoplus_{i+j=m}\mathcal{B}^{ij}_{c, \infty}\]
is also a one dimensional vector space. Let $\psi$ be the composition of the following maps:
\[ \xymatrix{
\bigoplus\limits_{i+j=m}{}{\mathcal{B}^{ij}_{c, \infty}} \ar[r]^\lambda & \bigoplus\limits_{i+j=m}{}{\mathcal{B}^{ij}_{ \infty} }\ar[r]^-{c} &\mathsf{Gr}(H^m(TM))\ar[r]^-{\tau} & H^m(TM) \ar[r]^-{\sigma^*}&H^m(M).
}\]
The map $c$ is the isomorphism given by the convergence of the spectral sequence. The map $\tau$ is the identification between
$H^m(M)$ and its associated graded, which exists because $H^m(TM)$ is one dimensional. The map $\sigma^*$ is the pull-back by the zero section. Since \[\mathcal{B}^{ij}_{c,0}=\Omega^{ji}_c(TM),\]
an element in ${\mathcal{B}^{ij}_{c, \infty}}$ is represented by a class $[\omega_{ij}]$ where $\omega_{ij}\in \Omega^{ji}_c(TM)$. Hence, the map $\psi$ is given by
\[ \psi[\omega_{ij}]=\sigma^*(\omega_{ij}).\]
This implies that \[\psi: {\mathcal{B}^{ij}_{c, \infty}}\rightarrow H^m(M)\] is the zero map if $i>0$. Since the maps $c, \tau $ and $\sigma^*$ are isomorphisms we obtain that the map
\[ \lambda: {\mathcal{B}^{ij}_{c, \infty}} \rightarrow {\mathcal{B}^{ij}_{ \infty}}\]
is zero for $i>0$. Furthermore, the following diagram is commutative
\[ \xymatrix{
\bigoplus\limits_{i+j=m}{}{\mathcal{B}^{ij}_{c, \infty}} \ar[r]^\lambda \ar[d]^-{c}& \bigoplus\limits_{i+j=m}{}{\mathcal{B}^{ij}_{ \infty} }\ar[d]^-{c}\\
\mathsf{Gr}(H^m_c(TM)) \ar[r]^-{\lambda} \ar[d]^-{\tau}& \mathsf{Gr}(H^m(TM))\ar[d]^-{\tau} \\
H_c^m(TM) \ar[r]^-{\lambda}& H^m(TM)
}\]
therefore the bottom map is zero if an only if the top map is zero if an only if the map
\[\lambda: {\mathcal{B}^{0m}_{c, \infty}}\rightarrow {\mathcal{B}^{0m}_{ \infty}}\]
is zero. Lemma \ref{in} then implies the claimed result.
\end{proof}

\subsection{A lemma about invariant forms}
Let $V$ be a finite dimensional real vector space. The vector space \[TV=V \oplus V\]
has a natural complex structure given by
\[ I(v,w):=(-w,v)\]
and therefore it is naturally a complex vector space. The group $\mathsf{Aff}(V)$ acts naturally on $TV$ via the derivative by the formula
\[ g \cdot (v,w)=(g(v), A(w))\]
for $g \in \mathsf{Aff}(V)$ of the form $g(v)=Av+b$. $S^1 \subseteq \MC$ acts on $TV$ by rotation
 \[
 \theta(v,w)=(\cos\theta v-\sin\theta w,\sin\theta v+\cos\theta w),\quad e^{i\theta}\in S^1.
 \]
 For each $\theta$, we denote by $TV_\theta$ the vector space $TV$ with the action of $\mathsf{Aff}(V)$ twisted by $\theta$
\[ g_\theta(v,w)=\theta^{-1} g (\theta(v,w)).\]

\begin{lemma}\label{invariant}
The vector space of invariant forms $\Omega^{0m}(TV_\theta)^{\mathsf{SAff}(V)}$ is independent of $\theta$ and has dimension one.
\end{lemma}
\begin{proof}
Let us first prove that for $\theta=1$ the space  $\Omega^{0m}(TV_\theta)^{\mathsf{SAff}(V)}$ is one dimensional.
Fix a basis $v_1, \dots, v_m$ for $V$ which determines coordinates: $\{x_1,\dots, x_m, y_1,\dots,y_m\}$ on $TV$.
We will show that $\Omega^{0m}(TV_\theta)^{\mathsf{SAff}(V)}$ is generated by $\omega=dy_1 \wedge \dots \wedge dy_m$. Let us first prove that $\omega \in \Omega^{0m}(TV_\theta)^{\mathsf{SAff}(V)}$.  We compute
\[ g^*(\omega)(X_1, \dots ,X_m)=\omega(A X_1, \dots , AX_m)=\mathsf{det}(A) \omega(X_1, \dots , X_m)=\omega(X_1, \dots ,X_m).\]
We conclude that $g^*(\omega)=\omega$. It remains to prove that the set of invariant forms
is one dimensional. Fix a point $(v,w) \in TV$ with $w \neq 0$ and an invariant form $\eta$. Then since $\Lambda^m(V)$ is a one dimensional vector space we know that at that point $\eta$ is a multiple of $\omega$. We may assume that they coincide at that point $(v,w)$. Now the set $W$ of points $(x,y) \in TV$ such that $y \neq 0$ open and dense, and moreover the group $\mathsf{SAff}(V)$ acts transitively on it. Since both $\eta$ and $\omega$ are invariant it follows that
they coincide on $W$. Since $W$ is dense, they coincide on $TV$. We still need to prove that  $\Omega^{0m}(TV_\theta)^{\mathsf{SAff}(V)}$ is independent of $\theta$.

 Given $g(v)= Av+b\in \mathsf{SAff}(V)$ the derivative
\[ Dg(p):V \oplus V \rightarrow V \oplus V\]
is a complex linear map and therefore

\[ g_\theta^*(\omega)(X_1, \dots ,X_m)=\omega(A X_1, \dots , AX_m)=\mathsf{det}(A) \omega(X_1, \dots , X_m)=\omega(X_1, \dots ,X_m).\]
This shows that $\omega$ belongs to $\Omega^{0m}(TV_\theta)^{\mathsf{SAff}(V)}$. Since $TV_\theta$ is equivariantly isomorphic to $TV_1$ we know that $\Omega^{0m}(TV_\theta)^{\mathsf{SAff}(V)}$ is also one dimensional and therefore:
\[\Omega^{0m}(TV_\theta)^{\mathsf{SAff}(V)}=\Omega^{0m}(TV_1)^{\mathsf{SAff}(V)}.\]

\end{proof}
\subsection{Spectral sequences of sheaves}

Let $M$ be an affine manifold so that $TM$ has a natural para-hypercomplex structure and for each $z \in S^1$ there is a decomposition of the de-Rham complex
\[
\xymatrix{
\vdots& \vdots& \vdots&\\
\Omega^{02}(TM) \ar[r]^{d^z_H}\ar[u]^{d^z_V}& \Omega^{12}(TM)  \ar[r]^{d^z_H}\ar[u]^{d^z_V}&\Omega^{22}(TM)\ar[r]^{d^z_H}\ar[u]^{d^z_V} &\cdots\\
\Omega^{01}(TM) \ar[r]^{d^z_H}\ar[u]^{d^z_V}&\Omega^{11}(TM) \ar[r]^{d^z_H}\ar[u]^{d^z_V}& \Omega^{21}(TM) \ar[r]^{d^z_H}\ar[u]^{d^z_V}&\cdots \\
\Omega^{00}(TM) \ar[r]^{d^z_H}\ar[u]^{d^z_V}& \Omega^{10}(TM) \ar[r]^{d^z_H}\ar[u]^{d^z_V}&\Omega^{20}(TM) \ar[r]^{d^z_H}\ar[u]^{d^z_V}&\cdots
}\]
We will denote by $\mathcal{B}(z)^{pq}_{c,r}$ and $\mathcal{B}(z)^{pq}_{r}$ the spectral sequences associated to the
vertical filtrations of these double complexes with compact and arbitrary support, respectively.
\begin{definition}
For each $z \in S^1$ we define the complex of sheaves
\[L_z^p:= \mathsf{ker} (d^z_H: \Omega^{0p} \rightarrow \Omega^{1p}),\]
with differential $d^z_V$.
\end{definition}
\begin{exercise}
Let $M$ be a connected manifold. Show that there is a canonical difeomorphism between the tangent bundle of the universal cover of $M$ and the universal cover of the tangent bundle of $M$
\[ \widetilde{TM}\cong T(\widetilde{M}).\]
\end{exercise}
\begin{definition}
Let $M$ be an affine manifold and $\varphi: \widetilde{M}\rightarrow V$ a developing map. The submersion
$ \psi: T(\widetilde{M})\times S^1 \rightarrow TV=V \oplus V$ is defined by
\[\psi( \gamma, e^{i\theta}):=R_\theta \circ D\varphi(\gamma),\]
where $R_\theta$ is the rotation matrix
\[
R_\theta :=\left(\begin{matrix}
\cos(\theta)& \sin(\theta)\\
\sin(\theta) & \cos(\theta)
\end{matrix}\right) \in \mathsf{GL}(V\oplus V).
\]
\end{definition}

\begin{definition}
We define the complex of sheaves $\mathcal{L}^p$ over $TM \times S^1$ as the descent of the $\pi_1(TM)$ equivariant
complex of sheaves $\psi^{-1}(L^p_1)$ over $T(\widetilde{M})\times S^1$.  Also, we define the bicomplex of sheaves
$\Omega^{pq}_{TM \times S^1}$ as the descent of the equivariant bicomplex $\psi^{-1}(\Omega^{pq}(TV), d_H +d_V)$.
\end{definition}

\begin{remark}
The complex of sheaves $L^p_1$ over $TV$ resolves the constant sheaf $\R_{TV}$ and is quasi-isomorphic to the total complex of the double complex complex $\Omega^{pq}(TV)$. Since the functor $\psi^{-1}$ is exact we conclude that
the complex $\mathcal{L}^p$ resolves the constant sheaf and is quasi-isomorphic to the total complex of $\Omega^{pq}_{TM \times S^1}$. For each $z \in S^1$ the restriction of $\mathcal{L}^p$ to $TM \times \{z\}$ is the complex of sheaves
$L^p_z$ and the restriction of $\Omega^{pq}(TV)$ is the bicomplex $(\Omega^{pq}(TM), d^z_H +d^z_V)$.
\end{remark}
\begin{definition}
Consider the second projection $ p_2:TM \times S^1 \rightarrow S^1$.
\begin{itemize}
\item We will denote by $\mathcal{E}^{pq}_{c,r}$ the spectral sequence computing
\[Rp_{2!}(\R)=Rp_{2!}(\mathcal{L^\bullet}).\]
\item We will denote by $\mathcal{E}^{pq}_{r}$ the spectral sequence computing
\[Rp_{2\ast}(\R)=Rp_{2\ast}(\mathcal{L^\bullet}).\]
\item We will denote by $\lambda:\mathcal{E}^{pq}_{c,r} \rightarrow \mathcal{E}^{pq}_{r}$ the morphism of spectral sequences of sheaves induced by the morphism of functors $Rp_{2!} \rightarrow Rp_{2\ast}$.
\end{itemize}
\end{definition}

\begin{lemma}\label{factors}
Let $M$ be a closed affine manifold. Then the map
\[ \lambda: \mathcal{B}^{0m}_{c, \infty}\rightarrow \mathcal{B}^{0m}_\infty\]
factors through the vector space $(\mathcal{E}_{\infty}^{0m})_{z=1}$.
\end{lemma}
\begin{proof}
By Remark \ref{com}, we know that for each $z \in S^1$ there are natural maps of spectral sequences:
\[\alpha: (\mathcal{E}^{pq}_{c,r})_z \rightarrow \mathcal{B}(z)^{pq}_{c,r}; \,\,\,\,\beta: (\mathcal{E}^{pq}_{r})_z \rightarrow \mathcal{B}(z)^{pq}_{r} .\]
Moreover, $\alpha$ is an isomorphism and $\beta$ is injective. Let us fix $z=1$. Since these morphisms are natural we obtain that the map
\[\lambda: \mathcal{B}^{0m}_{c, \infty}\rightarrow \mathcal{B}^{0m}_\infty\]
can be factored as the following composition:
\[\xymatrix{
 \mathcal{B}^{0m}_{c, \infty}\ar[r]^{\alpha^{-1}}&(\mathcal{E}^{0m}_{c,\infty})_{z=1}\ar[r]^{\lambda}&(\mathcal{E}^{0m}_{\infty})_{z=1}\ar[r]^\beta& \mathcal{B}^{0m}_\infty
}.\]
\end{proof}

\begin{lemma}\label{U}
The exists a unique open set $j:U \hookrightarrow S^1$ such that the sheaf $\mathcal{E}^{m0}_\infty$ is isomorphic to
the sheaf $j_!(\R_U)$.
\end{lemma}
\begin{proof}
The sheaf $\mathcal{E}^{m0}_\infty$ is a constructible subsheaf of $\R_{S^1}$ and therefore its support $U$ is open in $S^1$ and there is an isomorphism $\mathcal{E}^{m0}_\infty \simeq j_!(R_U)$. \end{proof}

\begin{proposition}\label{vanishingt}
Let $M$ be an closed special affine manifold. Then
\[ \mathcal{E}_{\infty}^{0m}=0.\]
\end{proposition}
\begin{proof}
It suffices to show that the inclusion morphism $\mathcal{E}^{m0}_\infty \hookrightarrow \R_{S^1}$ is an isomorphism.
By Lemma \ref{invariant} we can choose a nonzero element
$\omega \in \Omega^{0m}(TV_\theta)^{\mathsf{SAff}(V)}$ which defines a global section $\overline{\omega}$ of the sheaf
$\Omega^{0m}_{TM \times S^1}$. Since $\omega$ is closed, so is $\overline{\omega}$. This defines a global section of the
sheaf $\mathcal{E}^{m0}_1$ and therefore also a section $[\omega]$ of its quotient $\mathcal{E}^{m0}_\infty$. Let us proof that $[\omega]$ is nonzero. For each $z \in S^1$ consider the composition
\[ \xymatrix{
\mathcal{E}^{m0}_\infty(S_1) \ar[r] &(\mathcal{E}^{m0}_\infty)_z \ar[r]^\beta& \mathcal{B}(z)^{m0}_\infty,}
\]
where the first morphism is given by taking the stalk of a section. This map sends $[\omega]\in \mathcal{E}^{m0}_\infty(S_1) $
to $[[\omega]]\in  \mathcal{B}(z)^{m0}_\infty$. Let us set $z=\sqrt{-1}$.  Then  $\mathcal{B}(\sqrt{-1})^{m0}_{\infty}\simeq \R $ and $[[\omega]]$ is a generator. We conclude that $[\omega]\in \mathcal{E}^{m0}_\infty(S_1) $ is a nonzero section.
By Lemma \ref{U} we know that $\mathcal{E}^{m0}_\infty$ is of the form $j_!(\R_U)$, so it only can admit a nonzero section if
$U=S^1$ and  the map $\mathcal{E}^{m0}_\infty \hookrightarrow \R_{S^1}$ is an isomorphism.

\end{proof}

\begin{theorem}[Klingler]
The Euler characteristic of a closed special affine manifold is zero.
\end{theorem}
\begin{proof}
By Lemma \ref{factors} and Proposition \ref{vanishingt} we know that the map
\[ \lambda: \mathcal{B}^{0m}_{c, \infty}\rightarrow \mathcal{B}^{0m}_\infty\]
is zero.  Proposition \ref{spectralt} implies that in this case the Euler characteristic vanishes.
\end{proof}

\appendix

\section{Some differential geometry}
\subsection{Connections}

Given a smooth function $f=(f^{1},\dots,f^{m}):M\rightarrow\mathbb{R}^{m}$ and
a vector field $X\in\mathfrak{X}(M)$ it makes sense to consider the derivative
of $f$ in the direction of $X$
\[
X(f)=(X(f^{1}),\dots,X(f^{m}))=Df (X).
\]

On the other hand, if $\alpha\in\Gamma(E)$ is a section of a vector bundle
$E$, there is no natural way to differentiate $\alpha$ in the direction of a
vector field. A connection on a vector bundle $E$ is a rule that prescribes
how to differentiate sections of $E$ in the direction of vector fields.

\begin{definition}
Let $\pi:E\rightarrow M$ be a vector bundle. A connection $\nabla$ on $E$ is a
linear map
\[
\nabla:\mathfrak{X}(M)\otimes\Gamma(E)\rightarrow\Gamma(E);\quad
(X,\alpha)\mapsto\nabla_{X}\alpha
\]
such that for any smooth function $f\in C^{\infty}(M)$, $X\in\mathfrak{X}(M)$
and $\alpha\in\Gamma(E)$ the following two conditions are satisfied:
\begin{enumerate}
\item \[\nabla_{fX}\alpha=f\nabla_{X}\alpha,\]
\item \[\nabla_{X}\left(  f\alpha\right)  =\left(  X(f)\right)  \alpha
+f\nabla_{X}\alpha.\]
\end{enumerate}
\end{definition}

\begin{exercise}
Show that in the case where $E=M \times\mathbb{R}^{m}$ is the trivial bundle,
the directional derivative described above is a connection on $E$.
\end{exercise}

\begin{definition}
If $E$ is a vector bundle with connection, we will say that a section
$\alpha\in\Gamma(E)$ is covariantly constant if $\nabla_{X}(\alpha)=0$ for all
vector fields $X\in\mathfrak{X}(M)$.
\end{definition}

Let us now consider the case $E=TM$ and describe how a connection is expressed
in local coordinates $\varphi=(x^{1},\dots,x^{m})$. The Christoffel symbols
$\Gamma_{ij}^{k}:M\rightarrow\mathbb{R}$ are smooth functions determined by
the condition \[\nabla_{\partial_{i}}{\partial_{j}}=\sum_{k}%
\Gamma_{ij}^{k}\partial_{k}.\]
The connection $\nabla$ is
determined by the Christoffel symbols. Given vector fields
$X=\sum_{i}a^{i}\partial_{i}$ and $Y=\sum_{j}b^{j}\partial_{j}$ one computes
\begin{align*}
\nabla_{X}Y  &  =\sum_{i}a^{i}\nabla_{\partial_{i}}\left(  \sum_{j}%
b^{j}\partial_{j}\right)  =\sum_{i,j}a^{i}\nabla_{\partial_{i}}\left(
b^{j}\partial_{j}\right) \\
&  =\sum_{i,j}a^{i}\left(  \frac{\partial b^{j}}{\partial x^{i}}\partial
_{j}+b^{j}\nabla_{\partial_{i}}\partial_{j}\right) \\
&  =\sum_{i,j}a^{i}\left(  \frac{\partial b^{j}}{\partial x^{i}}\partial
_{j}+b^{j}\sum_{k}\Gamma_{ij}^{k}\partial_{k}\right) \\
&  =\sum_{i,j}a^{i}\frac{\partial b^{j}}{\partial x^{i}}\partial_{j}%
+\sum_{i,j}a^{i}b^{j}\sum_{k}\Gamma_{ij}^{k}\partial_{k}\\
&  =\sum_{k}\left(  \sum_{i}a^i\frac{\partial b^{k}}{\partial x^{i}}\partial
_{k}+\sum_{i,j}\Gamma_{ij}^{k}b^{j}a^{i}\right)  \partial_{k}.
\end{align*}


A Riemannian metric $g$ on a manifold $M$ induces a connection, called
the \emph{Levi-Civita Connection}, on the tangent bundle $TM$.

\begin{definition}
Let $\nabla$ be a connection on $TM$. The torsion of $\nabla$ is the function
\begin{align*}
T: \mathfrak{X}(M) \otimes\mathfrak{X}(M) \rightarrow\mathfrak{X}(M);
\quad(X,Y)\mapsto\nabla_{X} Y -\nabla_{Y} X - [X,Y].
\end{align*}

\end{definition}

\begin{exercise} Show that given vector fields $X,Y\in\mathfrak{X}(M) $, the
torsion satisfies

\begin{itemize}
\item Linearity with respect to functions: \[T(fX,Y)=fT(X,Y);\quad
T(X,fY)=fT(X,Y).\]

\item Skewsymmetry: \[T\left(  X,Y\right)  +T\left(  Y,X\right)  =0.\]
\end{itemize}

\end{exercise}

The previous exercise implies that one can view the torsion as a tensor

\[
T\in\Omega^{2}(M,TM)=\Gamma(\Lambda^{2}(T^{\ast}M)\otimes TM),
\]
defined by
\[
T(p)(v,w)=\nabla_{X}Y(p)-\nabla_{Y}X(p)-[X,Y](p),
\]
for any choice of vector fields $X,Y$ such that $X(p)=v$ and $Y(p)=w$.

\begin{definition}
A connection on $TM$ is called symmetric if its torsion is zero.
\end{definition}

\begin{exercise}
Show that a connection $\nabla$ is symmetric if and only if for any choice of
coordinates, the Christoffel symbols satisfy $\Gamma_{ij}^{k}=\Gamma_{ji}%
^{k}.$
\end{exercise}

\begin{definition}
A connection on a riemannian manifold $\left(  M,g\right)  $ is
compatible with the metric if
\[
{X}(g(Y,Z))=g\left(  \nabla_{X}Y,\text{ }Z\right)  +g\left(  Y,\text{ }%
\nabla_{X}Z\right)  ,
\]
for all $X,Y,Z\in\mathfrak{X}(M).$
\end{definition}

\begin{theorem}
[Levi-Civita]\label{4Teo3}Let $\left(  M,g\right)  $ be a riemannian
manifold. There exists a unique symmetric connection $\nabla$ which is
compatible with the metric. Moreover, this connection satisfies
\begin{align}
g\left(  Z,\nabla_{Y}X\right)   &  =\frac{1}{2}\left(  Xg\left(  Y,Z\right)
+Yg\left(  Z,X\right)  -Zg\left(  X,Y\right)  \right. \nonumber\\
&  \left.  -g\left(  \left[  X,Z\right]  ,Y\right)  -g\left(  \left[
Y,Z\right]  ,X\right)  -g\left(  \left[  X,Y\right]  ,Z\right)  \right)  .
\label{4ec9}%
\end{align}

\end{theorem}

\begin{proof}
Any connection compatible with the metric satisfies
\[
Xg\left(  Y,Z\right)  =g\left(  \nabla_{X}Y,Z\right)  +g\left(  Y,\nabla
_{X}Z\right)  ,
\]%
\[
Yg\left(  Z,X\right)  =g\left(  \nabla_{Y}Z,X\right)  +g\left(  Z,\nabla
_{Y}X\right)
\]
\[
Zg\left(  X,Y\right)  =g\left(  \nabla_{Z}X,Y\right)  +g\left(  X,\nabla
_{Z}Y\right)  ,
\]
Adding the first two equations, subtracting the third and using the symmetry one obtains
\begin{align*}
&  Xg\left(  Y,Z\right)  +Yg\left(  Z,X\right)  -Zg\left(  X,Y\right) \\
&  =g\left(  \left[  X,Z\right]  ,Y\right)  +g\left(  \left[  Y,Z\right]
,X\right)  +g\left(  \left[  X,Y\right]  ,Z\right)  +2g\left(  Z,\nabla
_{Y}X\right)  ,
\end{align*}
which implies
\begin{align*}
g\left(  Z,\nabla_{Y}X\right)   &  =\frac{1}{2}\left(  Xg\left(  Y,Z\right)
+Yg\left(  Z,X\right)  -Zg\left(  X,Y\right)  \right. \\
&  \left.  -g\left(  \left[  X,Z\right]  ,Y\right)  -g\left(  \left[
Y,Z\right]  ,X\right)  -g\left(  \left[  X,Y\right]  ,Z\right)  \right)  .
\end{align*}
Since the metric is nondegenerate, this implies uniqueness.
In order to prove existence we define  $\nabla_{Y}%
X$
to be the unique vector field that satisfies Equation (\ref{4ec9}). In order to prove that $\nabla$ defined in this way is a connection, the only nontrivial statement is
\[\nabla_X (fY)=f \nabla_X Y +X(f) Y.\]
For this we compute
\begin{align*}
g\left(  Z,\nabla_{Y}\left(  fX\right)  \right)   &  =\frac{1}{2}\left(
fXg\left(  Y,Z\right)  +Yg\left(  Z,fX\right)  -Zg\left(  fX,Y\right)  \right.
\\
&  \left.  -g\left(  \left[  fX,Z\right]  ,Y\right)  -g\left(  \left[
Y,Z\right]  ,fX\right)  -g\left(  \left[  fX,Y\right]  ,Z\right)  \right)  .
\end{align*}
Using the equations
\[
Yg\left(  Z,fX\right)  =\left(  Yf\right)  g\left(  Z,X\right)  +fYg\left(
Z,X\right)  ,
\]%
\[
Zg\left(  fX,Y\right)  =\left(  Zf\right)  g\left(  X,Y\right)  +fZg\left(
X,Y\right)  ,
\]%
\[
g\left(  \left[  fX,Z\right]  ,Y\right)  =fg\left(  \left[  X,Z\right]
,Y\right)  -\left(  Zf\right)  g\left(  X,Y\right)
\]
\[
g\left(  \left[  fX,Y\right]  ,Z\right)  =fg\left(  \left[  X,Y\right]
,Z\right)  -\left(  Yf\right)  g\left(  X,Z\right)
\]
we obtain
\begin{align*}
g\left(  Z,\nabla_{Y}\left(  fX\right)  \right)   &  =fg\left(  Z,\nabla
_{Y}X\right)  +\frac{1}{2}\left(  2\left(  Yf\right)  g\left(  Z,X\right)
\right) \\
&  =g\left(  Z,f\nabla_{Y}X+\left(  Yf\right)  X\right)  .
\end{align*}
We leave it as an exercise to the reader to prove that $\nabla$ is symmetric and compatible with the metric.
\end{proof}

The connection described above is called the Levi-Civita
connection on $(M,g)$.

\subsection{Geodesics and the exponential map}

Here we will explain the notions of parallel transport and geodesics. It will be convenient to
first discuss some natural operations on vector bundles and connections.

\begin{exercise}
Let $\nabla$ be a connection on $\pi:E\rightarrow M$ and $f:N\rightarrow M$ a
smooth function. Then there exists a unique connection $f^{\ast}(\nabla)$ on
$f^{\ast}(E)$ such that for any $\alpha\in\Gamma(E)$, $X\in\mathfrak{X}(N)$
and $Y\in\mathfrak{X}(M)$ with $Df(p)(X(p))=Y(f(p))$ the following holds:
\begin{equation}
f^{\ast}(\nabla)_{X}(f^{\ast}(\alpha))(p)=\nabla_{Y}(\alpha)(f(p)).
\label{pull0}%
\end{equation}

\end{exercise}

Recall that we say that a section $\alpha\in\Gamma(E)$ of a vector bundle with
connection is covariantly constant if $\nabla_{X}(\alpha)=0,$ for any vector
field $X\in\mathfrak{X}(M)$. By imposing this conditions on vector bundles
over an interval one obtains the notion of parallel transport along a path.

\begin{proposition}
Let $\nabla$ be a connection on a vector bundle $\pi:E\rightarrow I$, where
$I=[a,b]$ is an interval. Given a vector $v\in E_{a}$ there exists a unique
covariantly constant section $\alpha\in\Gamma(E)$ such that $\alpha(a)=v.$
Moreover, the function $P_{a}^{b}:E_{a}\rightarrow E_{b}$ given by $P_{a}%
^{b}(v)=\alpha(b)$ is a linear isomorphism. The function $P_{a}^{b}$ is called
the parallel transport of the connection $\nabla$.
\end{proposition}

\begin{proof}
Since all vector bundles over an interval are trivializable, we may choose a frame $\{\alpha_1,\dots, \alpha_k\}$ for $E$.
There exists a one form $\theta \in \Omega^1(I, \mathsf{End}(E))$ such that:
\[ \nabla_X (\alpha_i)= \theta(X, \alpha_i).\]
Let us fix $v = \sum_i \lambda_i \alpha_i(a) \in E_a$. A section $ \alpha= \sum_i f_i \alpha_i$  is covariantly constant if it satisfies the differential equation
\[ \sum_i \nabla_{\partial_t} (f_i \alpha_i)=0,\]
which is equivalent to
\[ \sum_i \frac{\partial f_i}{\partial t} \alpha_i+ f_i\theta( \partial_t, \alpha_i)=0.\]
The Picard-Lindel\"of theorem guarantees the existence and uniqueness of a solution of this equation. In order to show that $P_a^b$ is linear it is enough to observe that if $ \alpha$ and $\beta$ are covariantly constant, so is $\alpha + \beta $. It remains to show that  $P_a^b$ is an isomorphism. Suppose that  $v\in E_a$  is such that $P_a^b(v)=0$.
By symmetry we know that there exists a unique section  $\alpha \in \Gamma(E)$ such that $\alpha(b)=0$. This section is the zero section and we conclude that $v=0$.
\end{proof}

\begin{definition}
Let $\nabla$ be a connection on $\pi:E\rightarrow M$ and $\gamma
:[a,b]\rightarrow M$ a smooth curve. The parallel transport along $\gamma$
with respect to $\nabla$ is the linear isomorphism
\[
P_{\nabla}(\gamma):E_{\gamma(a)}\rightarrow E_{\gamma(b)};\quad P_{\nabla
}(\gamma)(v)=P_{a}^{b}(v),
\]
where $P_{a}^{b}$ denotes the parallel transport associated with the vector
bundle $\gamma^{\ast}(E)$ over the interval $I=[a,b]$ with respect to the
connection $\gamma^{\ast}(\nabla)$.
\end{definition}

\begin{lemma}
Let $\gamma:[a,c]\rightarrow M$ be a curve and $b\in(a,c)$. Set $\mu
=\gamma|_{[a,b]};\quad\sigma=\gamma|_{[b,c]}.$ Then $P_{\nabla}(\gamma
)=P_{\nabla}(\sigma)\circ P_{\nabla}(\mu).$
\end{lemma}

\begin{proof}
It is enough to observe that if $\alpha \in \Gamma(\gamma^*(E))$ is covariantly constant then $ \alpha\vert_{[a,b]}$ and $\alpha\vert_{[b,c]}$ are also covariantly constant.
\end{proof}

\begin{exercise}
Show that parallel transport is parametrization invariant. That is, if $\nabla$ is a
connection on $\pi:E\rightarrow M$, $\gamma:[a,b]\rightarrow M$ is a curve and
$\varphi:[c,d]\rightarrow\lbrack a,b]$ is an orientation preserving
diffeomorphism then $P_{\nabla}(\gamma)=P_{\nabla}(\gamma\circ\varphi).$
\end{exercise}

\begin{definition}
Let $\nabla$ be a connection on $TM$. A curve $\gamma:[a,b]\rightarrow M$ is called a geodesic if the
section $\gamma' \in \Gamma( \gamma^*(TM))$ is covariantly constant with respect to the
connection $\gamma^{\ast}(\nabla)$.
\end{definition}

In local coordinates $\varphi=(x^{1},\dots,x^{m})$ where $\gamma=(u_{1}%
,\dots,u_{m})$ and $\nabla$ has Christoffel symbols $\Gamma_{ij}^{k}$ one has
$\gamma^{\prime}(t)=\sum_{i}u_{i}^{\prime}(t)\partial_{i},$
and the geodesic equation takes the form
\begin{align*}
\gamma^{\ast}(\nabla)_{\partial_{t}}(\gamma^{\prime}(t))  &  =\sum_{i}%
\gamma^{\ast}(\nabla)_{\partial_{t}}(u_{i}^{\prime}(t)\partial_{i})\\
&  =\sum_{i}\Big(u_{i}^{\prime\prime}(t)\partial_{i}+u_{i}^{\prime}%
(t)\gamma^{\ast}(\nabla)_{\partial_{t}}\partial_{i}\Big)\\
&  =\sum_{i}\Big(u_{i}^{\prime\prime}(t)\partial_{i}+u_{i}^{\prime}(t)\sum
_{j}u_{j}^{\prime}(t)\nabla_{\partial_{j}}\partial_{i}\Big)\\
&  =\sum_{i}\Big(u_{i}^{\prime\prime}(t)\partial_{i}+u_{i}^{\prime}%
(t)\sum_{j,k}u_{j}^{\prime}(t)\Gamma_{ij}^{k}\partial_{k}\Big).
\end{align*}
We conclude that $\gamma$ is a geodesic precisely when it satisfies the system
of differential equations
\begin{equation}
u_{i}^{\prime\prime}(t)+\sum_{j,k}u_{j}^{\prime}(t)u_{k}^{\prime}%
(t)\Gamma_{kj}^{i}=0,\,\,  \forall i.
\end{equation}

\begin{example}
On Euclidian space $\mathbb{R}^{m}$ the Christoffel symbols are $\Gamma
_{ij}^{k}=0,$ and therefore the differential equation for a geodesic is just
$u_{i}^{\prime\prime}(t)=0.$ We conclude that geodesics in euclidean space are
straight lines.
\end{example}

\begin{theorem}
Let $\nabla$ be a connection on $TM$. Given $v\in T_{p}M$, there
exists an interval $\left(  -\epsilon,\epsilon\right)  $ for which there is a
unique geodesic $\gamma:\left(  -\epsilon,\epsilon\right)  \rightarrow M$ such
that $\gamma\left(  0\right)  =p$ and $\gamma^{\prime}(0)=v$.
\end{theorem}

\begin{proof}
Let $\varphi =(x^1,\dots, x^m)$  be local coordinates such that $\varphi(p)=0$. We  write $\gamma(t)=(u_1(t),\dots,u_m(t))$ and want to solve the system of equations
\[ u''_i(t) + \sum_{j,k} u'_j(t) u'_k(t) \Gamma_{kj}^i  =0.\]
This is a second order ordinary differential equation. The existence and uniqueness of solutions is guaranteed by the Pickard-Lindel\"of theorem.
\end{proof}

\begin{definition}
Let $\nabla$ be a connection on $TM$ and $ p \in M$. We define $A_p \subseteq T_pM$ as follows:
\[\ A_p:= \{ v\in T_pM: \text{there exists a geodesic } \gamma_v: [-1,1] \rightarrow M, \text{ with } \gamma_v(0)=p \text{ and }  \gamma_v'(0)=v\}.\]
The exponential map is defined by
\[ \mathsf{Exp_p}: A_p \rightarrow M; \, v \mapsto \gamma_v(1).\]
\end{definition}
The proof of the following theorem can be found in any text on riemannian geometry, for example \cite{riemannian}.
\begin{theorem}
Let $\nabla$ be a connection on $TM$ and $ p \in M$. The domain $A_p$ of the exponential map contains an open
neighborhood around $0 \in T_pM$. Moreover, the derivative of the exponential map at $0$ is the identity and therefore the exponential map is a local diffeomorphism.
\end{theorem}

\subsection{Curvature and the Chern-Gauss-Bonnet Theorem}

The Gauss-Bonnet theorem provides a formula for the Euler characteristic of a closed oriented surface $\Sigma$
\[ \chi(\Sigma)= \frac{1}{2\pi} \int_\sigma K dA.\]

Here $K$ denotes the Gaussian curvature of $\Sigma$ associated to a riemannian metric and $dA$ is the volume form determined by the metric and the orientation. This formula is remarkable because while the left hand side evidently depends only on the topology of $\Sigma$, the right hand side is a priori a geometric quantity. The extension of this formula to higher dimensions had to wait until the language of differential geometry was developed and Chern \cite{Chern} proved his generalised version of the Gauss-Bonnet theorem. The first obstruction that needs to be overcome in order to state a correct generelisation is to find a replacement for the gaussian curvature. This is provided by the Riemann curvature tensor.

\begin{definition}
Let $\nabla$ be a connection on $TM$. The riemannian curvature of $\nabla$ is the map
\[ R: \mathfrak{X}(M) \otimes \mathfrak{X}(M) \otimes \mathfrak{X}(M) \rightarrow \mathfrak{X}(M),\]
given by
\[ R_\nabla (X,Y,Z):= \nabla_X (\nabla_Y Z)- \nabla_Y (\nabla_X Z)- \nabla_{[X,Y]}Z.\]
A simple computation shows that this map is $C^\infty(M)$-linear in all the components and skew symmetric in $X$ and $Y$. Therefore, it defines a tensor
\[ R_\nabla \in \Omega^2(M, \mathsf{End}(TM))\]
which is called the Riemann curvature tensor. A connection $\nabla$ is said to be flat if $R_\nabla=0.$
\end{definition}

\begin{exercise}
Show that if $\nabla$ is the Levi-Civita connection of a riemannian manifold then for any pair of tangent vectors
$v,w \in T_pM$, the map
\[ R_\nabla(v,w,-): T_pM \rightarrow T_pM\]
is antisymmetric, i.e.
\[  \langle R_\nabla(v,w,z), z \rangle =0.\]
This is one of the Bianchi identities.
\end{exercise}

\begin{exercise} Let $ \mathfrak{so}(2n) $ be the Lie algebra of skew symmetric matrices.
The Pfaffian polynomial: \[\mathsf{Pf}: \mathfrak{so}(2n) \rightarrow \R,\] is defined by the formula

\[ \mathsf{Pf}(A)= \frac{1}{n!2^n }\sum_{\sigma \in S_{2n}} sg(\sigma) \prod_{i=1}^n a_{\sigma(2i-1),\sigma(2i)}.\]
Show that for $B \in \mathsf{End}(\R^{2n})$ and $ A \in \mathfrak{so}(2n)$ the following holds:
\begin{itemize}
\item[(a)] \[ \mathsf{Pf}(B A B^t)= \det(B) \mathsf{Pf}(A).\]
\item[(b)] \[ \mathsf{Pf}(A)^2= \det(A).\]
\end{itemize}
In particular, if $B \in SO(n)$ then
\[ \mathsf{Pf}(B A B^{-1})= \mathsf{Pf}(A).\]
\end{exercise}

\begin{exercise}
Let $V$ be a real vector space of dimension $2n$ and $ \langle \,, \rangle$ an inner product in $V$. Let $\mathfrak{so}(V)$ be the space of antisymmetric endomorphisms of $V$. That is
\[ \mathfrak{so}(V):=\{ \phi: V \rightarrow V: \langle v , \phi(v)\rangle=0\}.\]
Let $\mathsf{P}(A)=\mathsf{P}(a_{ij}): \mathfrak{so}(2n) \rightarrow \R$ be a polynomial which is invariant under the action of $SO(2n)$ i.e. such that
\[ P(BAB^{-1})=P(A)\]
for all $B \in SO(2n)$. Fix an orthonormal basis $\{ e_1, \dots, e_{2n}\}$ for $V$. For any $\omega \in \Lambda^2 (V^*) \otimes \mathfrak{so}(V)$ define $\omega_{ij} \in \Lambda^2(V^*)$ by the formula
\[ \omega_{ij}(v,w):=\langle \omega(v,w)(e_i), e_j \rangle.\]
Show that $\mathsf{P}(\omega):= \mathsf{P}(\omega_{ij})\in \Lambda(V^*)$ does not depend on the choice of orthonormal basis.
\end{exercise}

The previous exercise shows that for any riemannian manifold $(M,g)$ of dimension $2n$ there is a well defined form $\mathsf{Pf}(K)\in \Omega^{2n}(M)$ which is defined by
\[ \mathsf{Pf}(K)(p):= \mathsf{Pf}(K(p))\in \Lambda^{2n}(T^*_pM).\]
This differential form is what needs to be integrated over $M$ to obtain the Euler characteristic:

\begin{theorem}[Chern-Gauss-Bonnet]
Let $(M,g)$ be a closed oriented riemannian manifold of dimension $d=2m$ and $R$ the curvature of the Levi-Civita connection. Then
\[ \chi(M)=\Big(\frac{1}{2\pi}\Big)^n \int_M \mathsf{Pf}(K).\]
\end{theorem}

Of course, in case the dimension of $M$ is odd, the Euler characteristic vanishes by Poincar\'e duality.

\section{Spectral Sequences}
This section contains a brief introduction to spectral sequences.
\begin{definition}
  A cohomologically graded spectral sequence of real vector spaces $E=\{E_r\}$ is a sequence of bigraded vector spaces $E_r=\oplus E_r^{p,q}$ together with differentials of degree $(r,-r+1)$
  \[
  d_r:E_r^{p,q}\to E_r^{p+r,q-r+1}
  \]
  such that $d_r^2=0$ and $E_{r+1}=H^*(E_r)$.
\begin{equation}
\begin{tikzpicture}[descr/.style={fill=white}]
\matrix(m)[matrix of math nodes, row sep=3em, column sep=2.8em,
text height=1.5ex, text depth=0.25ex]
{E_{r}^{p-r,q}&E_{r}^{p,q}&E_{r}^{p+r,q}&\\&E_{r}^{p,q-r+1}&E_{r}^{p+r,q-r+1}&E_{r}^{p+2r,q-r+1}\\};
\path[->,font=\scriptsize]
(m-1-1) edge node[right] {$d_r$} (m-2-2);
\path[->,font=\scriptsize]
(m-1-2) edge node[right] {$d_r$} (m-2-3);
\path[->,font=\scriptsize]
(m-1-3) edge node[right] {$d_r$} (m-2-4);
\path[dashed,-,font=\scriptsize]
(m-1-1) edge node[below] {} (m-1-2);
\path[dashed,-,font=\scriptsize]
(m-1-2) edge node[below] {} (m-1-3);
\path[dashed,-,font=\scriptsize]
(m-2-2) edge node[below] {} (m-2-3);
\path[dashed,-,font=\scriptsize]
(m-2-3) edge node[below] {} (m-2-4);
\path[dashed,-,font=\scriptsize]
(m-1-2) edge node[below] {} (m-2-2);
\path[dashed,-,font=\scriptsize]
(m-1-3) edge node[below] {} (m-2-3);
\end{tikzpicture}
\end{equation}
  A morphism between spectral sequences $f:E\to F$ is a sequence of linear maps $f_r:E_r\to F_r$ compatible with the differentials such that $f_{r+1}$ is the morphism induced by $f_r$, i.e.
  \[
  f_{r+1}=H^*(f_r):E_{r+1}=H^*(E_r)\to F_{r+1}=H^*(F_{r}).
  \]
  Spectral sequences form a category.
\end{definition}
We usually think of the terms in a spectral sequence as pages in a book where the index $r$ denotes the page and $(p,q)$ are coordinates in that page.\\\\
Let $Z_1=\ker(d_1)$ and $B_1=\text{Im} (d_1)$. Since $E_{2}=H^*(E_1)=Z_1/B_1$, we may regard the kernel and image of $d_2$ as quotients $\ker(d_2)=Z_2/B_1$ and $\text{Im}(d_2)=B_2/B_1$, whence $B_1\subset B_2\subset Z_2\subset Z_1$. By induction we have $E_{r+1}=Z_r/B_r$, $\ker(d_{r+1})=Z_{r+1}/B_r$, $\text{Im}(d_{r+1})=B_{r+1}/B_r$ and $B_r\subset B_{r+1}\subset Z_{r+1}\subset Z_r$. This chain of inclusions allows us to define
\[
B_\infty=\cup B_r,\qquad Z_\infty=\cap Z_r.
\]
Clearly $B_\infty\subset Z_\infty$.
\begin{definition}
  The term at infinity of a spectral sequence is $E_\infty=B_\infty/Z_\infty$.
\end{definition}
A simple example of a spectral sequence is a cochain complex $d:C^\bullet\to C^{\bullet+1}$. Taking $E_1=C$ with the differential of the complex and $E_2=H^*(C)$ with zero differential. Notice that $E_r=E_2$ for $r\geq2$, therefore the sequence stabilizes at the second page making $E_\infty=E_2$.\\
Another source of examples of spectral sequences are exact couples.
\begin{definition}
  An exact couple is a couple of bigraded vector spaces $D,E$ with linear maps $i,j,k$ such that the following diagram is exact:
  \[
\begin{tikzpicture}[descr/.style={fill=white}]
\matrix(m)[matrix of math nodes, row sep=3em, column sep=2.8em,
text height=1.5ex, text depth=0.25ex]
{D&&D\\&E&\\};
\path[->,font=\scriptsize]
(m-1-1) edge node[above] {$i$} (m-1-3);
\path[->,font=\scriptsize]
(m-1-3) edge node[right] {$j$} (m-2-2);
\path[->,font=\scriptsize]
(m-2-2) edge node[left] {$k$} (m-1-1);
\end{tikzpicture}
\]
\end{definition}
\begin{lemma}
Given an exact couple $\{D,E,i,j,k\}$ we can produce another exact couple as follows. Let $d=jk$, then $d^2=jkjk=0$ which makes $d$ a differential on $E$, we will denote classes in homology by overlines. Let $E'=H^*(E)$ and $D'=i(D)$, $i'$ the restriction of $i$ to $D'$, $j'(i(x))=\overline{j(x)}$ and $k'$ the map induced by $k$. Then $\{D',E',i',j',k'\}$ is an exact couple.
\end{lemma}
\begin{proof}
  The proof is elementary. First we check that $j'$ and $k'$ are well defined:
  \begin{itemize}
    \item For $j'$ suppose that $i(x)=i(y)$, which means that $x-y\in\ker(i)=\text{Im}(k)$, so $x-y=k(e)$ for some $e\in E$. Applying $j$ we get $j(x-y)=jk(e)=d(e)$, therefore $\overline{j(x)}=\overline{j(y)}$.
    \item For $k'$ we want to check that $k'(\bar e)=k(e)\in D'=\ker(j)$. Since $e$ defines a homology class we have $d(e)=jk(e)=0$, so $k(e)\in\ker(j)$.
  \end{itemize}
  Now we go for exactness:
  \begin{itemize}
    \item At $\cdots\overset{k'}{\to}D'\overset{i'}{\to}\cdots$ Clearly $i'(k'(\bar e))=ik(e)=0$. Now take $i(x)\in D'$ such that $i'(i(x))=0$, by exactness in $D$ there is $e\in E$ such that $k(e)=i(x)$, so $d(e)=jk(e)=ji(x)=0$, whence $e$ defines a class in homology and $k'(\bar e)=k(e)=i(x)$.
    \item At $\cdots\overset{i'}{\to}D'\overset{j'}{\to}\cdots$ For $i(x)\in D'$ we have $j'(i'(i(x)))=\overline{j(i(x))}=0$. Now suppose $j'(i(x))=\overline{j(x)}=0$, then $j(x)=d(e)=jk(e)$ for some $e\in E$, meaning that $x-k(e)\in\ker(j)=\text{Im}(i)$. Let $y\in D$ such that $i(y)=x-k(e)$, then $i(x)=i(x)-ik(e)=ii(y)\in \text{Im}(i')$.
    \item At $\cdots\overset{j'}{\to}E'\overset{k'}{\to}\cdots$ Take $i(x)\in D'$, then $k'(j'(i(x)))=k'(\overline{j(x)})=kj(x)=0$. Now if $k'(\bar e)=k(e)=0$, then there is $x\in D$ with $j(x)=e$, then clearly $j'(i(x))=\overline{j(x)}=\bar e$.
  \end{itemize}
\end{proof}
Consider now an initial exact couple $\{D_0,E_0,i_0,j_0,k_0\}$ such that $i_0,j_0,k_0$ have degrees $(-1,1)$, $(0,0)$, $(0,1)$ respectively. Applying the previous construction consecutively we get a sequence of exact couples $\{D_r,E_r,i_r,j_r,k_r\}$ with degrees $(-1,1)$, $(r,-r)$, $(0,1)$. Notice that $d_r=j_rk_r$ is a differential on $E_r$ of degree $(r,-r+1)$, and by construction $E_r=H^*(E_{r-1})$, making $\{E_r\}$ a spectral sequence.\\
Given a filtered cochain complex we can produce an exact couple which in turn gives rise to a spectral sequence. Suppose $\{C^\bullet,d\}$ is a cochain complex with a decreasing filtration $\cdots,F^pC\supset F^{p+1}C\supset\cdots$ of subcomplexes. The graded complex associated to this filtration is actually bigraded, let $E_0$ denote the graded complex, then $E_0^{p,q}=F^pC^{p+q}/F^{p+1}C^{p+q}$. The differential on $E_0$ is $\bar d:E_0^{p,q}\to E_0^{p,q+1}$ which is induced by $d$. Notice that $\bar d$ has degree $(0,1)$. Now consider the short exact sequence of cochain complexes
\[
0\to F^{p+1}C^\bullet\overset{i}{\longrightarrow}F^{p}C^\bullet\overset{j}{\longrightarrow} E_0^{p,\bullet}\to 0
\]
By the snake lemma we get a long exact sequence
\[
\cdots\to H^n(F^{p+1}C^\bullet)\overset{i_*}{\longrightarrow}H^n(F^{p}C^\bullet)\overset{j_*}{\longrightarrow} H^n(E_0^{p,\bullet})\overset{k}{\longrightarrow}H^{n+1}(F^{p+1}C^\bullet)\to\cdots
\]
Finally let us define $D^{p,q}_1=H^{p+q}(F^{p+1}C^\bullet)$ and $E_1^{p,q}=H^{p+q}(E_0^{p,\bullet})$. Notice that with this convention the maps $i_*,j_*,k$ have degrees $(-1,1)$, $(1,-1)$ and $(0,1)$. Clearly $\{D_1,E_1,i_*,j_*,k\}$ is an exact couple, thus we have a spectral sequence associated to the filtration.\\
We will give an explicit formula for the terms of $\{E_r\}$. We start by defining the cycles up to filtration, this is, for $r\geq0$
\begin{equation}\label{cyclesfilt}
A^{p,q}_r=\{x\in F^pC^{p+q}\quad |\quad d(x)\in F^{p+r}C^{p+q+1}\}\end{equation}
\begin{lemma}
  The terms $E_r$ of the spectral sequence are
  \begin{equation}\label{spectral}
    E_r^{p,q}C=\frac{\big(A^{p,q}_r+F^{p+1}C^{p+q}\big)}{\big(d(A^{p-r+1,q+r-2}_{r-1})+F^{p+1}C^{p+q}\big)}
  \end{equation}
\end{lemma}
\begin{proof}
  First notice that equation (\ref{spectral}) for $r=0$ yields $E_0^{p,q}=F^pC^{p+q}/F^{p+1}C^{p+q}$. We will prove that projection $j:F^pC\to F^pC/F^{p+1}C$ induces a surjective map $\bar j: A^{p,q}_r+F^{p+1}C^{p+q}\to Z_{r-1}^{p,q}$. Furthermore, we will see that $\bar j(d(A^{p-r+1,q+r-2}_{r-1})+F^{p+1}C^{p+q})=B_{r-1}^{p,q}$ and the inverse image of $B_{r-1}^{p,q}$ is precisely $d(A^{p-r+1,q+r-2}_{r-1})+F^{p+1}C^{p+q}$. Thus $\bar j$ induces an isomorphism
  \[
  \frac{A^{p,q}_r+F^{p+1}C^{p+q}}{d(A^{p-r+1,q+r-2}_{r-1})+F^{p+1}C^{p+q}}\cong\frac{Z_{r-1}^{p,q}}{B_{r-1}^{p,q}}=E_r^{p,q}.
  \]
  For this purpose we prove the following equations:
  \begin{equation}
    Z_r=k^{-1}(\text{Im}(i_{r})),\qquad B_r=j(\ker(i_r)).
  \end{equation}
  For the first equation take $z\in Z_r$. $d_r(\bar z)=j_rk_r(\bar z)=0$ is equivalent to $k_r(\bar z)\in\ker(j_r)=\text{Im}(i_r)$. Since $k_r(z)$ is induced by the map $k$, we have $k_r(\bar z)=k(z)$, which means $k_r(\bar z)\in\text{Im}(i_r)$ if and only if $z\in k^{-1}(\text{Im}(i_{r}))$.\\
  Now take an element $j_rk_r(z)\in B_r$, then $k_r(z)\in \text{Im}(k_r)=\ker(i_r)$. $j_r$ is just a restriction of the map $j$, hence $B_r=j(\ker(i_r))$.\\\\
  Let us define the map $\bar j$. Take $x\in A^{p,q}_r+F^{p+1}C^{p+q}$ and consider its image $j(x)\in E_0^{p,q}$. Since $d(x)\in F^{p+r}C$ we have $d(j(x))=0$, therefore $j(x)$ defines an element in $H^*(E_0^{p,q})$ which we denote $\bar x$. Notice that $k(\bar x)=\overline{d(x)}\in H^*(F^{p+1}C)$, it is not always zero since $d(x)$ is not necessarily a boundary in the complex $F^{p+1}C$, furthermore, $k(\bar x)\in\text{Im}(i_{r-1})$, whence $\bar x\in k^{-1}(\text{Im}(i_{r-1}))=Z_{r-1}$.\\
  $\bar j$ is surjective: take an element $w\in Z^{p,q}_{r-1}$, hence $k(w)=i_{r-1}(y)$ for some $y\in H^{p+q+1}(F^{p+r}C)$, we will also denote a representative of $y$ by $y$. Next suppose that $j(x_1)$ is a representative of $w$, this means that $d(x_1)=k(w)$, therefore $d(x_1)$ and $y$ represent the same class in $H^{p+q+1}(F^pC)$, hence there is an $x_2\in F^pC^{p+q}$ such that $d(x_1)-y=d(x_2)$. Taking $x=x_1+x_2\in A^{p,q}_r+F^{p+1}C^{p+q}$ we get that $j(x)$ represents $w$, and $d(x)=y$, making $\bar j(x)=w$.\\
  Let us check that $\bar j(d(A^{p-r+1,q+r-2}_{r-1})+F^{p+1}C^{p+q})\subset B_{r-1}^{p,q}$. Notice that the components in $F^{p+1}C^{p+q}$ maps to zero, therefore it is enough to check that $\bar jd(A^{p-r+1,q+r-2}_{r-1})\subset B_{r-1}^{p+q}$. Let $d(w)\in d(A^{p-r+1,q+r-2}_{r-1})$, which means that $w\in F^{p-r+1}C^{p+q-1}$. We have that $d(w)$ is not necessarily a boundary in $F^{p}C^{p+q}$, but becomes a boundary after applying the inclusion $F^{p}C^{p+q}\subset F^{p-r+1}C^{p+q}$, which means that its homology class goes to zero in $H^{p+q}(F^{p-r+1}C)$. We have proved that $\bar j(d(w))\in j(\ker (i_{r-1}))=B_{r-1}$.\\
  Finally let us check that $\bar j^{-1}(B_{r-1})=d(A^{p-r+1,q+r-2}_{r-1})+F^{p+1}C^{p+q}$. Suppose $\bar j(x)\in B_{r-1}$, this is, $\bar j(x)=j(u)$ for $u\in\ker(i_{r-1})$. Then $u$ must be a boundary in $F^{p-r+1}C$, this is, $u=d(z)$. Next notice that $j(x-d(z))=0$, so $x-d(z)\in F^{p+1}C$ and $x=d(z)+x-d(z)\in d(A^{p-r+1,q+r-2}_{r-1})+F^{p+1}C^{p+q}$.
\end{proof}
\begin{corollary}
  The term at infinity of the spectral sequence is
  \begin{equation}\label{inftyterm}
    E_\infty^{p,q}C=\frac{\big(A^{p,q}_\infty+F^{p+1}C^{p+q}\big)}{\big(d(A_{\infty})^{p,q}+F^{p+1}C^{p+q}\big)},
  \end{equation}
  where
  \begin{equation}\label{cyclesfilt2}
  A^{p,q}_\infty=\cap_{r\geq 1}A_r^{p,q} \qquad\text{and}\qquad\big(d(A_{\infty})^{p,q})=\cup_{r\geq 1}d(A^{p-r+1,q+r-2}_{r-1}).
  \end{equation}
\end{corollary}
Let $C$ be a cochain complex with a decreasing filtration of subcomplexes $\cdots \supset F^pC\supset F^{p+1}C\supset\cdots$. We denote by $E_rC$ the spectral sequence associated to the filtered complex. The cohomology of the complex may also be filtered by taking $F^pH^*(C)=\text{Im}(H^*(F^pC)\to H^*(C))$, whence we have a second spectral sequence $E_rH^*(C)$. The following theorem establishes a relation between both sequences.
\begin{theorem}
  If $C$ is a filtered complex such that $C=\cup_p F^pC$ and for every $n$ there exists an $m$ such that $F^mC^n=0$, then $E_\infty^{p,q} C=E_0^{p,q}H^*(C)$.
\end{theorem}
\begin{proof}
  We start by considering the induced filtrations on $Z$ and $B$, the cycles and boundaries of $C$:
  \[
  F^pZ^{s}=F^pC\cap\ker\{d:C^s\to C^{s+1}\},\qquad F^pB^s=F^pC\cap\text{Im}(d:C^{s-1}\to C^s),
  \]
  then $H(F^pC)=F^pZ/F^pB$.\\
  We will compute $E_0^{p,q}H^*(C)$. Recall that $E_0^{p,q}H^*(C)=F^pH^{p+q}(C)/F^{p+1}H^{p+q}(C)$ where $F^pH(C)=\text{Im}(H(F^pC)\to H(C))$. Notice that $F^pH^{p+q}(C)=(F^pZ+B)/B$, therefore we have
  \[
  E_0^{p,q}H^*(C)=\frac{F^pZ+B}{F^{p+1}Z+B}=\frac{F^pZ}{F^pZ\cap(F^{p+1}Z+B)}.
  \]
  Furthermore, we have the equality $F^pZ\cap(F^{p+1}Z+B)=F^pZ\cap(F^{p+1}C+F^pB)$. Indeed, let $x\in F^pZ\cap(F^{p+1}Z+B)$, then $x=a+d(b)$ with $a\in F^{p+1}Z$ and $d(b)\in B$. Clearly $a\in F^{p+1}C$ and $d(b)=x-a\in F^pC$, so $x=a+d(b)\in F^{p+1}C+F^pB$. Similarly, for $x\in F^pZ\cap(F^{p+1}C+F^pB)$ we have $x=a+d(b)$ with $a\in F^{p+1}C$ and $d(b)\in F^pB\subset B$. We have that $0=d(x)=d(a)+d^2(b)=d(a)$, so $a\in F^{p+1}Z$. We have proven that
  \[
  E_0^{p,q}H^*(C)=\frac{F^pZ}{F^pZ\cap(F^{p+1}C+F^pB)}=\frac{F^pZ+F^{p+1}C}{F^pB+F^{p+1}C}
  \]
  Recall the cycles up to filtration defined in (\ref{cyclesfilt}), notice that under our hypothesis, for $r$ large enough we have $F^pZ^*=A_r^{p,*}$, whence
  \[
  F^pZ^*+F^{p+1}C^*=A_r^{p,*}+F^{p+1}C^*=A_\infty^{p,*}+F^{p+1}C^*.
  \]
  On the other hand, for any $d(b)\in F^pB$ we must have some $s$ such that $b\in F^sC$. So $d(b)\in d(A^{s,p+q-s-1}_{p-s})\subset d(A_{\infty})^{p,q}$, therefore
  \[
  F^pB^*+F^{p+1}C^*=d(A_{\infty})^{p,*}+F^{p+1}C^*.
  \]
  Combining our previous findings we get
  \[
  E_0^{p,q}H^*(C)=\frac{\big(A^{p,q}_\infty+F^{p+1}C^{p+q}\big)}{\big(d(A_{\infty})^{p,q}+F^{p+1}C^{p+q}\big)}=E_\infty^{p,q}C.
  \]
\end{proof}
The following lemma provides an alternate formula for the terms of the spectral sequence that is more convenient for computational purposes.
\begin{lemma}
  The terms of the spectral sequence associated to a filtered complex $F^\bullet C$ are:
  \begin{equation}\label{terms}
    E_r^{p,q}C=\frac{A^{p,q}_r}{d(A^{p-r+1,q+r-2}_{r-1})+A^{p+1,q-1}_{r-1}},
  \end{equation}
  where $A^{p,q}_r$ is as in (\ref{cyclesfilt}).\\
  The terms at infinity are
  \begin{equation}\label{termsinfty}
    E_\infty^{p,q}C=\frac{A^{p,q}_\infty}{d(A_\infty)^{p,q}+A^{p+1,q-1}_{\infty}},
  \end{equation}
  where $A^{p,q}_\infty$ and $d(A_\infty)^{p,q}$ are as in (\ref{cyclesfilt2}).
\end{lemma}
\begin{proof}
We have the equation
\[
A^{p,q}_r\cap F^{p+1}C^{p+q}=\{x\in F^{p+1}C^{p+q}\quad|\quad d(x)\in F^{p+r}C^{p+q+1}\}=A^{p+1,q-1}_{r-1}.
\]
Since $d(A^{p-r+1,q+r-2}_{r-1})\subset A_r^{p,q}$, the isomorphism follows
\[
\frac{A_r^{p,q}+F^{p+1}C^{p+q}}{d(A^{p-r+1,q+r-2}_{r-1})+F^{p+1}C^{p+q}}=\frac{A^{p,q}_r}{d(A^{p-r+1,q+r-2}_{r-1})+A^{p,q}_{r}\cap F^{p+1}C^{p+q}}.
\]
Equation (\ref{termsinfty}) is proved in a similar fashion.
\end{proof}
\begin{example}\label{bete}
  Given a complex a complex $C$, a simple filtration is the so called bête filtration:
  \[
  (F^pC)^n=\left\{\begin{array}{ccc}0&\text{if}& n<p\\ C^n&\text{if}& n\geq p\end{array}\right.
  \]
  In this case the spectral sequence converges to the cohomology of the complex.
\end{example}
\begin{example}
As a more interesting example we compute the first terms of the spectral sequence associated to the vertical filtration of a double complex. Suppose $\Omega$ is a first quadrant double complex:
\[
\begin{CD}
  \vdots @.\vdots @.@.\\
  @AAA@AAA@.\\
  \Omega^{0,1}@>d^{0,1}_H>>\Omega^{1,1}@>>>\cdots\\
  @Ad^{0,0}_VAA@Ad^{1,0}_VAA@.\\
  \Omega^{0,0}@>d^{0,0}_H>>\Omega^{1,0}@>>>\cdots\\
\end{CD}
\]
The chain complex is $C^n=\oplus_{i+j=n} \Omega^{i,j}$ with differential $d^{i,j}=d^{i,j}_H+d^{i,j}_V$. The filtration is then
\[
F^rC^{p+q}=\bigoplus_{\underset{j\geq r}{i+j=p+q}}\Omega^{i,j}.
\]
The zeroth term of the sequence is
\[
E_0^{p,q}=\frac{F^pC^{p+q}}{F^{p+1}C^{p+q}}=\frac{\Omega^{q,p}\oplus\Omega^{q-1,p+1}\oplus\cdots\oplus\Omega^{p+q,0}}{\Omega^{q-1,p+1}\oplus\cdots\oplus\Omega^{p+q,0}}=\Omega^{q,p}
\]
Using (\ref{terms}) we compute the first term of the sequence:
\begin{align*}
  A^{p,q}_1=&\bigoplus_{\underset{j\geq p+1}{i+j=p+q}}\Omega^{i,j}\oplus\ker(d_H^{q,p})\\
  A^{p+1,q-1}_0=&\bigoplus_{\underset{j\geq p+1}{i+j=p+q}}\Omega^{i,j}\\
  A^{p,q-1}_0=&\bigoplus_{\underset{j\geq p}{i+j=p+q-1}}\Omega^{i,j}
\end{align*}
Notice that
\[d(A^{p,q-1}_0)+A^{p+1,q-1}_0=\bigoplus_{\underset{j\geq p+1}{i+j=p+q}}\Omega^{i,j}\oplus\text{Im}(d_H^{q-1,p})\]
therefore
\[
E_1^{p,q}=\frac{\left(\bigoplus_{\underset{j\geq p+1}{i+j=p+q}}\Omega^{i,j}\right)\oplus\ker(d_H^{q,p})}{\left(\bigoplus_{\underset{j\geq p+1}{i+j=p+q}}\Omega^{i,j}\right)\oplus\text{Im}(d_H^{q-1,p})}=\frac{\ker(d_H^{q,p})}{\text{Im}(d_H^{q-1,p})}.
\]
So the page $E_1$ of the spectral sequence is just the horizontal cohomology of the double complex which we denote $E_1^{p,q}=H_H^q(\Omega^{\bullet,p})$. Recall that the differential on $E_1$ is such that \[d_1:E_1^{p,q}=H_H^q(\Omega^{\bullet,p})\to E_1^{p+1,q}=H_H^q(\Omega^{\bullet,p+1}),\] therefore the second term of the spectral sequence is
\[
E_2^{p,q}=H_V^pH_H^q(\Omega).
\]
If we use the horizontal filtration instead, the computation follows with minor changes. The filtration is
\[
F^rC^{p+q}=\bigoplus_{\underset{i\geq r}{i+j=p+q}}\Omega^{i,j}.
\]
The zeroth term is $E_0^{p,q}=\Omega^{p,q}$, the first term is the vertical cohomology $E_1^{p,q}=H_V^q(\Omega^{p,\bullet})$ and the second term is $E_2^{p,q}=H_H^pH_V^q(\Omega)$.
\end{example}
\section{Sheaves}
\subsection{Basic Definitions and Results}
All spaces considered in this section are locally compact.
\begin{definition}
  Let $X$ be a topological space. A presheaf of vector spaces $F$ over $X$ consists of the following data:
  \begin{itemize}
    \item For every open set $U\subset X$ there is a vector space $F(U)$.
    \item For every inclusion $U\subset V$ there is a linear map $F(V)\to F(U)$ called a restriction map.
  \end{itemize}
  The elements of $F(U)$ are called sections of $F$ over $U$. If $U\subset V$ and $s\in F(V)$, we often denote the image of the restriction of $s$ to $U$ by $s|_U$. These restrictions must satisfy the following conditions:
  \begin{itemize}
    \item The restriction $F(U)\to F(U)$ must be the identity.
    \item If $U\subset V\subset W$ and $s\in F(W)$, then $(s|_V)|_U=s|_U$.
  \end{itemize}
\end{definition}
In the previous definition we may change the terms ``vector spaces" and ``linear maps" for objects and morphisms in any category. In fact, another way to define a presheaf is as follows: given a topological space $X$, we have the category $O(X)$ whose objects are open subsets of $X$ and whose morphisms are inclusions. If $C$ is any other category, a presheaf over $X$ with values in $C$ is a contravariant functor from $O(X)$ to $C$. We are mostly interested in the case where $C=\text{Vect}$, the category of real vector spaces.
\begin{definition}
  Let $F$ be a presheaf over $X$ and $x\in X$. The stalk of $F$ at $x$, denoted $F_x$, is the direct limit
  \[
  F_x=\lim_{\underset{U}{\longrightarrow}}F(U),
  \]
  where the limit is taken over all open sets containing $x$. If $s\in F(U)$, the image image of $s$ under the projection $F(U)\to F_x$ is denoted $s_x$ and is called the germ of $s$ at $x$.
\end{definition}
\begin{definition}
  A sheaf over $X$ is a presheaf $F$ that satisfies the following conditions:
  \begin{itemize}
    \item If $\{U_\alpha\}_\alpha$ is an open covering of an open set $U$ and $s,t$ are sections over $U$ such that $s|_{U_\alpha}=t|_{U_\alpha}$ for every $\alpha$, then $s=t$.
    \item If $\{U_\alpha\}_\alpha$ is an open covering of an open set $U$ and there are sections $s_\alpha\in F(U_\alpha)$ for every $\alpha$ such that $s_\alpha|_{U_\alpha\cap U_\beta}=s_\beta|_{U_\alpha\cap U_\beta}$ for each pair $\alpha,\beta$, then there is a section $s\in F(U)$ such that $s|_{U_\alpha}=s_\alpha$.
  \end{itemize}
  If $F,G$ are presheaves over $X$ with values in the same category, a morphism $\varphi:F\to G$ is a family of morphisms $\varphi_U:F(U)\to G(U)$ for every open set $U\subset X$. These morphisms must be compatible with the restriction morphisms, this is, if $U\subset V$, the following diagram must be commutative:
  \[\begin{CD}
      F(V) @>\varphi_V>> G(V)\\
      @VVV @VVV\\
      F(U) @>\varphi_U>> G(U)
      \end{CD}
      \]
  Categorically, a morphism of presheaves is a natural transformation $F\rightarrow G$. A morphism of sheaves is just a morphism as presheaves.
\end{definition}
We will denote the category of presheaves over $X$ by $\textit{PSh}(X)$ and the category of sheaves by $\textit{Sh}(X)$. Notice that
$\textit{Sh}(X)$ is a full subcategory of $\textit{PSh}(X)$. Given $x\in X$, taking the stalk of a sheaf at $x$ is a functor \[(-)_x:\textit{PSh}(X)\rightarrow \text{Vect}.\] If $U\subset X$ is an open set, we define the functor \[\Gamma(U;-):\textit{PSh}(X)\rightarrow \text{Vect}\] by $\Gamma(U;F)=F(U)$. This functor is called taking sections in $U$. In particular we call $\Gamma(X;-)$ the global sections functor.
\begin{proposition}
  If $F$ is a presheaf over $X$, then there is a sheaf $F^+$ over $X$ and a morphism of presheaves $\theta:F\to F^+$ unique up to isomorphism with the property that, for any other sheaf $G$ and morphism $\lambda:F\to G$, there is a unique morphism $\widehat\lambda:F^+\to G$ such that $\lambda=\widehat\lambda\circ\theta$.\\
  In other words, the map
  \[
  \text{Hom}_{\textit{Sh}(X)}(F^+,G)\to \text{Hom}_{\textit{PSh}(X)}(F,G)
  \]
  induced by $\theta$ is an isomorphism. This means that the functor $\textit{PSh}(X)\rightarrow \textit{Sh}(X)$ given by $F\mapsto F^+$ is a left adjoint of the forgetful functor.
\end{proposition}
The sheaf $F^+$ is called the sheaf associated to the presheaf $F$.\\

\begin{proof}
  We define $F^+(U)$ as the set of functions $\sigma:U\to \bigsqcup_{x\in U}F_x$ such that:
  \begin{enumerate}
    \item $\sigma(x)\in F_x$.
    \item For every $x\in U$ there is a neighbourhood $V\subset U$ of $x$ and a section $s\in F(V)$ such that for every $y\in V$, we have $\sigma(y)=s_y$.
  \end{enumerate}
  The first condition states that $\sigma$ is a section with respect to the natural projection $\bigsqcup_{x\in U}F_x\to U$. The morphism $\theta:F\to F^+$ is such that for $s\in F(U)$ and $x\in U$, $\theta(s)(x)=s_x$. If we endow $\bigsqcup_{x\in U}F_x$ with the finest topology such that all the sections of the form $\theta(s)$ are continuous, then the second condition on $\sigma$ is just a requirement of continuity. The properties are readily checked.
\end{proof}

Given a sheaf $F$ over $X$ and a continuous function $h:X\to Y$, we are interested in the following two ways of transferring the data of $F$ to $Y$ via $h$:
\begin{definition}
 The direct image of $F$ is a sheaf on $Y$ defined as follows: if $W\subset Y$ is an open set, we set $h_*F(W):=F(h^{-1}(W))$.
 \end{definition}
 \begin{definition}
 Given a section $s\in F(U)$, the support of $s$, denoted $\textit{supp}(s)$, is the closure $\overline{U-Z}$ where $Z$ is the union of all open sets $V\subset U$ such that $s|_V=0$. The direct image with compact support is a sheaf $h_!(F)$ on $Y$ defined as follows: for an open set $W\subset Y$, we set:
      \[
      h_!F(W)=\{s\in F(h^{-1}(W)): h|_{\textit{supp}(s)}:\textit{supp}(s)\to W\text{ is proper}\}.
      \]
\end{definition}

Given a sheaf $G$ over $Y$, we can also transfer the data of $G$ to $X$ via $h$.
\begin{definition}
  The sheaf $h^{-1}G$ over $X$ called the inverse image of $G$ is defined as follows: if $U$ is open in $X$, then
\[U\mapsto \lim_{\underset{V}{\longrightarrow}}G(V),\]
where $V$ ranges over the open neighbourhoods of $h(U)$. This procedure produces a presheaf, taking its associated sheaf we get the inverse image $h^{-1}G$. In the case where $i:Z\hookrightarrow X$ is a subspace and $F$ is a sheaf on $X$, we will denote the inverse image $i^{-1}F=F|_Z$, the restriction of $F$ to $Z$. Also, when taking sections of $F$ over $Z$ we write $\Gamma(Z;F)=\Gamma(Z;i^{-1}F)$. The sections of $F$ over $Z$ can also be obtained as a direct limit
\begin{equation}\label{sheafeq}
\Gamma(Z;F)=\lim_{\underset{U}{\longrightarrow}}\Gamma(U;F),
\end{equation}
where $U$ ranges over all open sets containing $Z$.
\end{definition}

When taking sections with compact support of a sheaf over a space, we will write $\Gamma_c$, this is,
\[
\Gamma_c(X,F)=\{s\in\Gamma(X,F)\quad\big|\quad \text{supp}(s)\text{ is compact.}\}.
\]
\begin{proposition}\label{sheafprop}
  Let $h:Y\to X$ be a continuous function and $G$ a sheaf over $Y$. Then for every $x\in X$, the are canonical morphisms
  \begin{equation}\label{eqsheaf}
  \alpha:(h_*G)_x\to \Gamma(h^{-1}(x);G|_{h^{-1}(x)})\quad\text{and}\quad\beta:(h_!G)_x\to \Gamma_c(h^{-1}(x);G|_{h^{-1}(x)})
  \end{equation}
  The morphism $\beta$ is actually an isomorphism.
\end{proposition}
\begin{proof}
  $\beta$ is injective: Take $s_x\in (h_!G)_x$ such that $\beta(s_x)=0$. A representative of $s_x$ is a section $s\in\Gamma(h^{-1}(U);G)$ such that $x\in U$ and $h|_{\textit{supp}(s)}:\text{supp}(s)\to U$ is proper. $\beta(s_x)=0$ implies that $\textit{supp}(s)$ and $h^{-1}(x)$ are disjoint, therefore $x\not\in h(\text{supp}(s))$. Since $h(\text{supp}(s))$ is closed, there is a neighbourhood of $x$ where the section vanishes, which means that $s_x=0$.\\

  $\beta$ is surjective: Take $s\in \Gamma_c(h^{-1}(x);G|_{h^{-1}(x)})$ and name $K=\text{supp(s)}$. By (\ref{sheafeq}), there is a neighbourhood $U$ of $K$ and a section $t$ over $U$ such that $t|_K=s|_K$. Restricting the open if necessary, we may assume that $t|_{U\cap h^{-1}(x)}=s|_{U\cap h^{-1}(x)}$. Next take a neighbourhood $V$ of $K$ such that $K\subset V\subset \bar V\subset U$. We have that $x\not\in h(\bar V\cap \text{supp}(t)\setminus V)$, there is an open neighbourhood $W$ of $x$ such that $h^{-1}(W)\cap\bar V\cap \text{supp}(t)\subset V$. Let us define $\widetilde s\in \Gamma(h^{-1}(W);G)$ as
  \[
  \widetilde s|_{h^{-1}(W)\setminus(\text{supp}(t)\cap\bar V)}=0,\qquad \widetilde s|_{h^{-1}(W)\cap V}=t|_{h^{-1}(W)\cap V}.
  \]
  The section $\widetilde s$ maps to $s$ via $\beta$.
\end{proof}
\subsection{Cohomology of Sheaves}
All sheaves considered in this section are over a space $X$.
\begin{definition}
  Let $O_X$ be a sheaf of rings with unit over $X$, we call $(X,O_X)$ a ringed space. A sheaf $A$ is a sheaf of modules over $O_X$ if for every open $U\subset X$, $\Gamma(U;A)$ is a module over $\Gamma(U;O_X)$ and for every inclusion of open sets $V\subset U$, the restriction map is a morphism of $\Gamma(U;O_X)$-modules.
\end{definition}
Consider a short exact sequence of sheaves of $O_X$-modules $0\to A\to B\to C\to 0$. The functor $\Gamma(X;-)$ is left exact, i.e. the sequence $0\to\Gamma(X;A)\to \Gamma(X;B)\to \Gamma(X;C)$ is exact. The map $\Gamma(X;B)\to \Gamma(X;C)$ is not surjective in general, however, there are functors $\{H^k(X;-)\}_{k\geq 0}$ that complete the sequence, this means that there is a long exact sequence
\begin{equation}\label{les}
0\to H^0(X;A)\to H^0(X;B)\to H^0(X;C)\to H^1(X;A)\to H^1(X;B)\to H^1(X;C)\to\cdots
\end{equation}
where $H^0(X;-)=\Gamma(X;-)$. The groups $H^k(X;A)$ are called the cohomology groups of $X$ with coefficients in $A$.\\

In order to construct the cohomology groups, we will use injective resolutions of sheaves. A sheaf $I$ is said to be injective if for every injective morphism $i:A\hookrightarrow B$ and every morphism $\phi:A\to I$ there is an extension of $\phi$ to $B$. That is, a morphism $\widetilde\phi:B\to I$ such that $\widetilde\phi \circ i=\phi$.\\

A resolution of $A$ is a sequence of sheaves $\{I^k\}_{k\geq0}$ and morphisms $\{d^k\}_{k\geq0}$ such that
\[
0\to A\longrightarrow I^0\overset{d^0}{\longrightarrow} I^1\longrightarrow\cdots
\]
is exact. An injective resolution is one where every sheaf $I^k$ is injective. We will denote resolutions by $0\to A\overset{d^\bullet}{\longrightarrow}I^\bullet$.

\begin{proposition}
  Every sheaf of $O_X$-modules has an injective resolution. This is often stated as the existence of enough injective objects in the category of sheaves of $O_X$-modules over $X$.
\end{proposition}
 Now let us take an arbitrary sheaf $A$ and an injective resolution $0\to A\overset{d^\bullet}{\longrightarrow}I^\bullet$. Consider the sequence obtained when we remove $A$
 \[
 0\to I^0\overset{d^0}{\longrightarrow} I^1\overset{d^1}{\longrightarrow}I^2\longrightarrow\cdots
 \]
 Applying the functor $\Gamma(X;-)$ we get a complex
 \[
 0\to \Gamma(X;I^0)\overset{\Gamma(X;d^0)}{\longrightarrow} \Gamma(X;I^1)\overset{\Gamma(X;d^1)}{\longrightarrow}\Gamma(X;I^2)\longrightarrow\cdots
 \]
 The homology groups of this complex are the cohomology groups of $X$ with values in $A$:
 \[
 H^k(X;A)=\frac{\ker(\Gamma(X;d^k))}{\text{Im}(\Gamma(X;d^{k-1}))}
 \]
 In particular we have $H^0(X;A)=\ker(\Gamma(X;d^0))=\Gamma(X;A)$ since $\Gamma(X;-)$ is left exact.\\

 The following proposition states that the cohomology groups depend only on $A$ and not on the choice of injective resolution.
 \begin{proposition}
   If $0\to A\overset{d^\bullet}{\longrightarrow}I^\bullet$ and $0\to A\overset{d'^\bullet}{\longrightarrow}I'^\bullet$ are injective resolutions of $A$, then there is a morphism of resolutions $\varphi^\bullet:I^\bullet\to I'^\bullet$ lifting the identity
   \[
   \begin{CD}
     0@>>> A@>>> I^0@>d^0>> I^1@>>> \cdots\\
     @. @V\text{Id}_AVV @V\varphi^0VV @V\varphi^1VV @.\\
     0@>>> A@>>> I'^0@>>d'^0> I'^1@>>> \cdots\\
   \end{CD}
   \]
   This morphism induces an isomorphism at the level of homology.
 \end{proposition}
 \begin{proposition}
   Let $0\to A\to B\to C\to 0$ be a short exact sequence of sheaves. If $0\to A\overset{d^\bullet_A}{\longrightarrow}I^\bullet_A$ and $0\to C\overset{d^\bullet_C}{\longrightarrow}I^\bullet_C$ are injective resolutions for $A$ and $C$ respectively, then $0\to B\longrightarrow I^\bullet_A\oplus I^\bullet_C$ is an injective resolution for $B$. Furthermore, $0\to I^\bullet_A\to I^\bullet_A\oplus I^\bullet_C\to I^\bullet_C\to 0$ is a short exact sequence of resolutions.
 \end{proposition}
 Applying the snake lemma to the sequence above one obtains (\ref{les}).\\

 \subsection{The Functors $Rh_*,Rh_!$ and the Induced Spectral Squences}
 The procedure applied to construct sheaf cohomology may be generalized to arbitrary left exact functors. Let $A$ be a sheaf over $X$ and $0\to A\overset{d^\bullet}{\longrightarrow}I^\bullet$ an injective resolution of $A$. Removing $A$ we obtain a (non-exact) sequence
 \[
 0\to I^0\overset{d^0}{\longrightarrow} I^1\overset{d^1}{\longrightarrow}I^2\longrightarrow\cdots
 \]
Now let $\mathcal{F}:\textit{Sh}(X)\to C$ be a left exact functor into an abelian category $C$. Applying $\mathcal{F}$ to the previous sequence we obtain a complex
\begin{equation}\label{complex}0\to \mathcal{F}(I^0)\overset{\mathcal{F}(d^0)}{\longrightarrow} \mathcal{F}(I^1)\overset{\mathcal{F}(d^1))}{\longrightarrow}\mathcal{F}(I^2)\longrightarrow\cdots\end{equation}
The $n$-th derived functor of $\mathcal{F}$, denoted $R^n\mathcal{F}$, is the $n$-th cohomology of the previous complex. Left exactness of $\mathcal{F}$ implies that $R^0\mathcal{F}=\mathcal{F}$.\\
If $h:X\to Y$ is a continuous function, then $h_*$ and $h_!$ are left-exact functors $\textit{Sh}(X)\to\textit{Sh}(Y)$. In particular we have the derived functors $Rh_*$ and $Rh_!$.
\begin{remark}
  The derived functors $Rh_*$ and $Rh_!$ may be computed using the spectral sequence associated to the bête filtration explained in Example \ref{bete} for the complex (\ref{complex}).
\end{remark}
\begin{proposition}
  The morphisms (\ref{eqsheaf}) from Proposition \ref{sheafprop} may be extended to morphisms
  \[
  \alpha^p:(R^ph_*(G))_x\to H^p(h^{-1}(x);G_{h^{-1}(x)})\quad\text{and}\quad\beta^p:(R^ph_!(G))_x\to H^p_c(h^{-1}(x);G_{h^{-1}(x)})
  \]
  As before, $\beta^p$ is an isomorphism.
\end{proposition}
This result is known as the proper base change theorem and may be consulted in \cite{Ka}.
\begin{remark}\label{com}
  The morphism $\alpha$ ($\beta$) induces a morphism between the spectral sequence computing the derived functors of $h_*$ (respectively $h_!$) over a point $x$ and the spectral sequence computing the sheaf cohomology $H^p(h^{-1}(x);G_{h^{-1}(x)})$ (respectively $H^p_c(h^{-1}(x);G_{h^{-1}(x)})$)
\end{remark}
 \subsection{Acyclic, Flabby, Soft, Fine and Constructible Sheaves}
 Injective sheaves are inadequate for computational purposes of sheaf cohomology. In this section we introduce sheaves that are better suited for the task. Once again, all sheaves considered are over $X$.
 A sheaf $L$ is called acyclic if $H^k(X;L)=0$ for all $k>0$.
 \begin{proposition}
   Let $0\to A\overset{d^\bullet}{\longrightarrow}L^\bullet$ be a resolution of $A$ by acyclic sheaves. Then there is a natural isomorphism
   \[
   \gamma:H^k(\Gamma(X;L^\bullet))\to H^k(X;A).
   \]
   This means that sheaf cohomology may be computed using acyclic resolutions instead of injective ones.
 \end{proposition}
 \begin{proof}
   We split the resolution into short exact sequences
   \[
   0\to M^{k}\to L^k\to M^{k+1}\to 0
   \]
   where $M^k=\ker(d^k)=\text{Im}(d^{k-1})$; in particular $M^0=A$. The induced long exact sequence in this case takes the form
   \[
   \cdots\to H^p(X;M^k)\longrightarrow H^p(X;L^k)\longrightarrow H^p(X;M^{k+1})\longrightarrow H^{p+1}(X;M^k)\longrightarrow H^{p+1}(X;L^k)\to\cdots
   \]
   Since $L^\bullet$ is acyclic, we find isomorphisms $H^p(X;M^{k+1})\cong H^{p+1}(X;M^k)$ for each $p>0$.\\

   When $p=0$ we have the following segment of the sequence
   \[
   0\to\Gamma(X;M^k)\longrightarrow \Gamma(X;L^k)\longrightarrow \Gamma(X;M^{k+1})\longrightarrow H^{1}(X;M^k)\to 0
   \]
   which gives an isomorphism
   \[
   H^{k+1}(\Gamma(X;L^\bullet))=\frac{\Gamma(X;M^{k+1})}{\text{Im}(\Gamma(X;L^k)\to \Gamma(X;M^{k+1}))}\cong H^{1}(X;M^k)
   \]
   Combining the previous isomorphisms we get
   \[
   H^{k+1}(\Gamma(X;L^\bullet))\cong H^{1}(X;M^k)\cong H^{2}(X;M^{k-1})\cong\cdots\cong H^{k+1}(X;M^0)=H^{k+1}(X;A)
   \]
 \end{proof}
 The following types of sheaves are examples of acyclic sheaves. The reader may consult the proofs in Theorems 5.5, 9.11 and 9.16 of \cite{Br}.
 \begin{definition}Let $A$ be a sheaf over $X$.
 \begin{itemize}
  \item  $A$ is flabby if for every open $U\subset X$, the restriction map $\Gamma(X;A)\to\Gamma(U;A)$ is surjective.
  \item $A$ is soft if for every closed $Z\subset X$, the restriction map $\Gamma(X;A)\to\Gamma(Z;A)$ is surjective.
  \item If $A$ is a sheaf of $O_X$-modules, $A$ is called fine if for every open cover $\{U_i\}$ of $X$ there is a partition of unity, i.e., a family of sections $s_i\in\Gamma(X;O_X)$ such that $\text{supp}(s_i)\subset U_i$ and $\sum_i s_i=1$. The sum is assumed to be locally finite, this means that for every $x\in X$ there is a neighbourhood where all but finite of the sections vanish.
   \end{itemize}
 \end{definition}
 \begin{example}
   Let $X$ be a smooth manifold. For every open $U\subset X$ let $C^{\infty}(U)$ denote the sheaf of smooth functions on $U$ with real values. $C^\infty$ is easily seen to be a sheaf of rings over $X$, hence $(X, C^\infty)$ is a ringed space. We know that the spaces of differential forms $\Omega^k(U)$ are modules over $C^\infty(U)$, so $\Omega^k$ can be seen to be a sheaf of $C^\infty$-modules. The existence of partitions of unity over $X$ guarantees that the sheaves $\Omega^k$ are fine sheaves.\\ \\
   Consider now the constant sheaf $\mathbb{R}_X$ that assigns the ring $\mathbb{R}$ to every non-empty open set of $X$ with restrictions given by the identity. The inclusion $\mathbb{R}\hookrightarrow C^\infty(U)=\Omega^0(U)$ as constant functions provides a map of sheaves $\mathbb{R}_X\to\Omega^0$. Then the de Rham differential gives us a sequence
   \[
   0\to\mathbb{R}_X\longrightarrow\Omega^0\overset{d}{\longrightarrow}\Omega^1\overset{d}{\longrightarrow}\cdots
   \]
   For every $x\in X$ we have a neighbourhood $U$ diffeomorphic to $\mathbb{R}^n$, by Poincare's Lemma we have that
   \[
   0\to\mathbb{R}\longrightarrow\Omega^0(U)\overset{d}{\longrightarrow}\Omega^1(U)\overset{d}{\longrightarrow}\cdots
   \]
   is exact. Then the sequence is exact at the level of sheaves and $0\to\mathbb{R}_X\overset{d}{\longrightarrow}\Omega^\bullet$ is an acyclic resolution for the constant sheaf. The cohomology of $X$ with values in the constant sheaf $\mathbb{R}_X$ is just the de Rham cohomology of $X$:
   \[
   H^\bullet(X;\mathbb{R}_X)=H^\bullet_{\text{DR}}(X).
   \]
 \end{example}
 Finally, we define constructible sheaves.
 \begin{definition}
 $A$ is said to be constructible if there is a partition $X=\bigsqcup X_i$ where the $X_i$ are locally closed subsets and $A|_{X_i}$ is a locally constant sheaf for every $i$.
 \end{definition}
\end{document}